\documentclass[aap]{imsart}

\usepackage{amsmath,amsthm,amssymb,latexsym,graphicx, hyperref,color, mathtools}
\usepackage[margin=1in]{geometry}

\startlocaldefs
\endlocaldefs

\newcommand{\NN}{{\mathbb N}}

\newcommand{\RR}{{\mathbb R}}
\newcommand{\ZZ}{{\mathbb Z}}

\newcommand{\abs}[1]{ \left| #1 \right|}

\newcommand{\del}{\partial}

\newcommand{\nv}{^{-1}}

\newcommand{\nor}[2]{\left\|#1\right\|_{#2}}
\newcommand{\oline}[1]{\overline{#1}}
\newcommand{\oo}{\infty}
\newcommand{\pars}[1]{\left(#1\right)}

\newcommand{\mcl}{\mathcal}
\newcommand{\mbb}{\mathbb}
\newcommand{\mbf}{\mathbf}

\newtheorem{lemma}{Lemma}

\newtheorem{theorem}{Theorem}

\newtheorem{definition}{Definition}

\numberwithin{equation}{section}
\numberwithin{lemma}{section}
\numberwithin{proposition}{section}
\numberwithin{theorem}{section}
\numberwithin{corollary}{section}
\numberwithin{definition}{section}

\begin{document}

\begin{frontmatter}

\title{Approximation schemes for viscosity solutions of fully nonlinear stochastic partial differential equations}
\runtitle{Approximation schemes for fully nonlinear SPDE}


\author{\fnms{Benjamin} \snm{Seeger}\thanksref{nsf}\ead[label=e1]{seeger@ceremade.dauphine.fr}}
\address{Place du Mar\'echal de Lattre de Tassigny\\ 75016 Paris, France\\ \printead{e1}}
\affiliation{Universit\'e Paris-Dauphine and Coll\`ege de France}

\thankstext{nsf}{Partially supported by the National Science Foundation Mathematical Sciences Postdoctoral Research Fellowship under Grant Number DMS-1902658} 

\runauthor{Benjamin Seeger}

\begin{abstract}
	The aim of this paper is to develop a general method for constructing approximation schemes for viscosity solutions of fully nonlinear pathwise stochastic partial differential equations, and for proving their convergence. Our results apply to approximations such as explicit finite difference schemes and Trotter-Kato type mixing formulas. The irregular time dependence disrupts the usual methods from the classical viscosity theory for creating schemes that are both monotone and convergent, an obstacle that cannot be overcome by incorporating higher order correction terms, as is done for numerical approximations of stochastic or rough ordinary differential equations. The novelty here is to regularize those driving paths with non-trivial quadratic variation in order to guarantee both monotonicity and convergence.
		
	We present qualitative and quantitative results, the former covering a wide variety of schemes for second-order equations. An error estimate is established in the Hamilton-Jacobi case, its merit being that it depends on the path only through the modulus of continuity, and not on the derivatives or total variation. As a result, it is possible to choose a regularization of the path so as to obtain efficient rates of convergence. This is demonstrated in the specific setting of equations with multiplicative white noise in time, in which case the convergence holds with probability one. We also present an example using scaled random walks that exhibits convergence in distribution.
\end{abstract}




\end{frontmatter}


\section{Introduction} \label{S:intro}

We construct numerical schemes to approximate viscosity solutions of fully nonlinear pathwise stochastic partial differential equations, and prove that they converge under quite general assumptions. Among the approximations that we study are finite-difference schemes and Trotter-Kato type product formulas. The former raise the possibility of numerical implementation, which we justify with precise error estimates in the first-order setting.

More precisely, given a finite horizon $T > 0$, we consider pathwise viscosity solutions of the initial value problem 
\begin{equation}\label{E:eq}
	du = F(D^2 u, Du)\; dt + \sum_{i=1}^m H^i(Du)\circ dW^i \quad \text{in } \RR^d \times (0,T] \quad \text{and} \quad u(\cdot,0) = u_0 \quad \text{in } \RR^d, 
\end{equation}
where $W = (W^1, W^2, \ldots, W^m) : [0,T] \to \RR^m$ is a continuous path and the initial datum $u_0: \RR^d \to \RR$ is bounded and uniformly continuous. The precise assumptions on $H = (H^1, H^2, \ldots, H^m): \RR^d \to \RR^m$ and $F: S^d \times \RR^d \to \RR$, where $S^d$ is the space of symmetric matrices, are specified later. We emphasize here that $F$ is assumed to be degenerate elliptic, that is, $F(X,p) \le F(Y,p)$ whenever $p \in \RR^d$ and $X,Y \in S^d$ satisfy $X \le Y$.

The technical assumptions and theorems are stated in full generality later in the paper. First, we describe the main results in a simplified context to provide a flavor for what is to follow.  Afterwards, we provide some background on the notion of pathwise viscosity solutions, the history of the study of the equation, and its applications. The Introduction concludes with a description of the organization of the rest of the paper.

We note that, in the sequel, the term ``classical viscosity theory'' refers to the Crandall-Ishii-Lions \cite{CIL} theory of viscosity solutions, which applies to \eqref{E:eq} when $W$ is continuously differentiable, or to the theory of equations with $L^1$-time dependence put forth by Ishii \cite{I} and Lions and Perthame \cite{LP} for Hamilton-Jacobi equations, and by Nunziante \cite{N} in the second order case, which includes \eqref{E:eq} when $W$ has bounded variation. 

\subsection{The main results}
Assume for now that $d = m = 1$, $F$ and $H$ are both smooth, and $F$ depends only on $u_{xx}$, so that \eqref{E:eq} becomes
\begin{equation}\label{E:simpleeqintro}
	du = F(u_{xx})\;dt + H(u_x)\circ dW \quad \text{in } \RR \times (0,T] \quad \text{and} \quad u(\cdot,0) = u_0 \quad \text{in } \RR,
\end{equation}
or, in the first order case, when $F \equiv 0$,
\begin{equation} \label{E:simpleeqintrofirstorder}
	du = H(u_x)\circ dW \quad \text{in } \RR \times (0,T] \quad \text{and} \quad u(\cdot,0) = u_0 \quad \text{in } \RR.
\end{equation}
Here and throughout the paper, the solutions of equations like \eqref{E:eq}, \eqref{E:simpleeqintro}, and \eqref{E:simpleeqintrofirstorder} are to be understood in the pathwise, or stochastic, viscosity sense (see Definitions \ref{D:smoothH} and \ref{D:nonsmoothH} below).

The approximations are constructed through the use of a scheme operator, which, for $h > 0$, $0 \le s \le t \le T$, and $\zeta \in C([0,T]; \RR)$, is a map $S_h(t,s; \zeta): BUC(\RR) \to BUC(\RR)$, whose properties will be made more precise in Section \ref{S:scheme}.  Here, $BUC(\RR^d)$ is the space of bounded, uniformly continuous functions on $\RR^d$. 

Throughout the paper, the symbol $\mcl P$ denotes a partition of $[0,T]$ and $\abs{ \mcl P}$ its mesh size, that is,
\[
	\mcl P := \{0 = t_0 < t_1 < \cdots <  t_N = T\} \quad \text{and} \quad \abs{ \mcl P} := \max_{n = 0, 1, \ldots, N-1} \pars{ t_{n+1} - t_n}.
\]

Given such a partition $\mcl P$ and a path $\zeta \in C([0,T]; \RR)$, usually a piecewise linear approximation of $W$, we first define the function $v_h(\cdot; \zeta, \mcl P)$ by
\begin{equation} \label{E:introapproxsolution}
	\begin{dcases}
		v_h(\cdot,0;\zeta, \mcl P ) := u_0, & \\[1.2mm]
		v_h(\cdot,t; \zeta, \mcl P) := S_h(t,t_n; \zeta)v_h(\cdot,t_n; \zeta, \mcl P) & \text{for } n = 0, 1, \ldots, N-1 \text{ and } t \in (t_n, t_{n+1}].
	\end{dcases}
\end{equation}

The strategy is to choose families of approximating paths $\{W_h\}_{h > 0}$ and partitions $\{\mcl P_h\}_{h > 0}$ satisfying
\begin{equation} \label{A:introapproximators}
	\lim_{h \to 0^+} \nor{W_h - W}{\oo} = 0 = \lim_{h \to 0^+} \abs{ \mcl P_h},
\end{equation}
in such a way that the function
\begin{equation} \label{E:introschemedef}
	u_h(x,t) := v_h(x,t; W_h, \mcl P_h)
\end{equation}
is an efficient approximation of the solution of \eqref{E:eq}.

As an example of the types of schemes studied in this paper, we consider here the following adaptation of the Lax-Friedrichs finite difference approximation, a formulation for which can be found in the work of Crandall and Lions \cite{CL} in the classical viscosity setting. 

For some $\epsilon_h > 0$, define
\begin{equation} \label{E:introLF}
	\begin{split}
	S_h(t,s; \zeta)u(x) &:= u(x) + H \pars{ \frac{u(x+h) - u(x-h)}{2h} } (\zeta(t) - \zeta(s)) \\
	&+ \pars{ F\pars{ \frac{u(x+h) + u(x-h) - 2u(x)}{h^2} } \right.\\
	&+ \left.\epsilon_h \pars{ \frac{u(x+h) + u(x-h) - 2u(x)}{h^2} } }(t-s).
	\end{split}
\end{equation}

The first result, which is qualitative in nature, applies to the simple setting above as follows:

\begin{theorem} \label{T:introLFresult}
	Assume that, in addition to \eqref{A:introapproximators}, $W_h$ and $\mcl P_h$ satisfy
	\[
		\abs{ \mcl P_h} \le \frac{h^2}{\nor{F'}{\oo}} \quad \text{and} \quad \epsilon_h := h \nor{\dot W_h}{\oo} \xrightarrow{h \to 0} 0.
	\]
	Then, as $h \to 0$, the function $u_h$ defined by \eqref{E:introschemedef} using the scheme operator \eqref{E:introLF} converges locally uniformly to the solution $u$ of \eqref{E:simpleeqintro}.
\end{theorem}

We obtain explicit error estimates for finite difference approximations of the stochastic Hamilton-Jacobi equation \eqref{E:simpleeqintrofirstorder}. The results below are stated for the following scheme, which is defined, for some $\theta \in (0,1]$, by
\begin{equation} \label{E:introLFfirstorder}
	\begin{split}
	S_h(t,s; \zeta)u(x) := u(x) &+ H \pars{ \frac{u(x+h) - u(x-h)}{2h} } (\zeta(t) - \zeta(s)) \\
	&+ \frac{\theta}{2} \pars{ u(x+h) + u(x-h) - 2u(x) }.
	\end{split}
\end{equation} 
Note that this corresponds to choosing $\epsilon_h := \frac{\theta h^2}{2(t-s)}$ in \eqref{E:introLF}. 

The main tool for proving rates of convergence is the following pathwise estimate. For the remaining results in the introduction, it is assumed that, for some $L > 0$, the initial datum $u_0$ is Lipschitz continuous with $\nor{u_0'}{\oo} \le L$.

\begin{theorem}\label{T:intropathwise}
There exists $C > 0$ depending only on the Lipschitz constant $L$ such that, if $h > 0$, $\zeta \in C([0,T], \RR)$ is piecewise linear over the partition $\mcl P$ such that
\[
	\max_{n = 0, 1, \ldots, N-1} \abs{ \zeta(t_{n+1}) - \zeta(t_n)} \le \frac{\theta}{\nor{H'}{\oo}} h,
\]
and $v$ solves \eqref{E:simpleeqintrofirstorder} with the path $\zeta$, then, for all $\epsilon > 0$,
\[
	\sup_{(x,t) \in \RR \times [0,T]} \abs{ v_h(x,t; \zeta, \mcl P) -  v(x,t)}  \le \frac{1}{\epsilon} \sum_{n=0}^{N-1} (t_{n+1} - t_n)^2 + C \sqrt{N} h + \max_{s,t \in [0,T]} \left\{ C\abs{ \zeta(s) - \zeta(t)} - \frac{ |s-t|^2}{2 \epsilon} \right\}. 
\]
\end{theorem}

The rates of convergence are then established by choosing families of paths $\{W_h\}_{h > 0}$ and partitions $\{ \mcl P_h\}_{h > 0}$ in order to optimize the estimate from Theorem \ref{T:intropathwise}. 

We do so first for an arbitrary, fixed continuous path $W$ with modulus of continuity $\omega: [0,\oo) \to [0,\oo)$. For $h > 0$, define $\rho_h$ implicitly by
\begin{equation} \label{A:ctspathCFL}
	\lambda := \frac{ (\rho_h)^{1/2} \omega( (\rho_h)^{1/2}) }{h} < \frac{\theta}{\nor{H'}{\oo}},
\end{equation}
and let the partition $\mcl P_h$ and path $W_h$ satisfy
\begin{equation} \label{A:introregularpathspartition} 
	\left\{
	\begin{split}
		&\mcl P_h := \{ n \rho_h \wedge T \}_{n \in \NN_0}, \; M_h := \lfloor (\rho_h)^{-1/2} \rfloor,\\[1.2mm]
		&\text{and, for } k \in \NN_0 \text{ and } t \in [kM_h\rho_h, (k+1)M_h\rho_h), \\[1.2mm]
		&W_h(t) := W(kM_h\rho_h) + \pars{ \frac{ W((k+1)M_h\rho_h) - W(kM_h\rho_h)}{M_h\rho_h} } \pars{ t - kM_h\rho_h}.
	\end{split}
	\right.
\end{equation}
	
\begin{theorem}\label{T:introLFresultfirstorder}
	There exists $C > 0$ depending only on $L$ such that, if $u_h$ is constructed using \eqref{E:introschemedef} and \eqref{E:introLFfirstorder} with $\mcl P_h$ and $W_h$ as in \eqref{A:ctspathCFL} and \eqref{A:introregularpathspartition}, and $u$ is the pathwise viscosity solution of \eqref{E:simpleeqintrofirstorder}, then
	\[
		\sup_{(x,t) \in \RR^d \times [0,T]} \abs{ u_h(x,t) - u(x,t)} \le C(1+T) \omega( (\rho_h)^{1/2}).
	\]
\end{theorem}

When $W$ is a Brownian motion, we study the problem from different points of view, depending on whether the focus is on almost-sure convergence or convergence in distribution.

As a special case of Theorem \ref{T:introLFresultfirstorder}, the approximating paths and partitions may be taken to satisfy \eqref{A:introregularpathspartition} with $\rho_h$ given by
\begin{equation} \label{A:Brownianregularpartition}
	\lambda := \frac{ (\rho_h)^{3/4} \abs{ \log \rho_h}^{1/2}}{h} < \frac{\theta}{\nor{H'}{\oo}}.
\end{equation}
Alternatively, the constructions may be achieved through the use of certain stopping times:
\begin{equation} \label{A:Brownianrandompartition}
	\left\{
	\begin{split}
		&T_0 := 0, \quad T_{k+1} := \inf\left\{ t > T_k : \max_{r,s \in [T_k, t]} |W(r) - W(s)| > \frac{h^{1/3}}{ |\log h|^{2/3} } \right\},\\[1.2mm]
		&W_h(t) := W(T_k) + \frac{W(T_{k+1}) - W(T_k)}{T_{k+1} - T_k}(t - T_k) \quad \text{for } t \in [T_k, T_{k+1}), \\[1.2mm]
		&M_h := \left \lceil \frac{ \nor{H'}{\oo} }{\pars{ h |\log h|}^{2/3}} \right \rceil, \\[1.2mm]
		&\text{and} \; \mcl P_h := \left\{ t_n := T_k + (n - k M_h) \frac{ T_{k+1} - T_k}{M_h} :k M_h \le n < (k+1) M_h, \; k \in \NN_0 \right\}.
	\end{split} 
	\right.
\end{equation}

The various definitions for $\mcl P_h$ and $W_h$ above, while technical, are all made with the same idea in mind, namely, to ensure that the approximation $W_h$ is ``mild'' enough with respect to the partition. In particular, for any consecutive points $t_n$ and $t_{n+1}$ of the partition $\mcl P_h$, and for sufficiently small $h$, the ratio
\[
	\frac{\abs{ W_h(t_{n+1}) - W_h(t_n)}}{h}
\]
should be less than some fixed constant. This is a special case of the kind of Courant-Lewy-Friedrichs (CFL) conditions required for the schemes in this paper, which are discussed in more detail in the following sections.

\begin{theorem} \label{T:introLFpathwiseBM}
	Suppose that $W$ is a Brownian motion, and assume either that $\mcl P_h$ and $W_h$ are as in \eqref{A:introregularpathspartition} with $\rho_h$ defined by \eqref{A:Brownianregularpartition}, or $\mcl P_h$ and $W_h$ are as in \eqref{A:Brownianrandompartition}. If $u_h$ is constructed using \eqref{E:introschemedef} and \eqref{E:introLFfirstorder}, and $u$ is the solution of \eqref{E:simpleeqintrofirstorder}, then there exists a deterministic constant $C > 0$ depending only on $L$ and $\lambda$ such that, with probability one,
	\[
		\limsup_{h \to 0} \sup_{(x,t) \in \RR^d \times [0,T]} \frac{ \abs{ u_h(x,t) - u(x,t)}}{h^{1/3} \abs{ \log h}^{1/3}} \le C(1+T).
	\]
\end{theorem}

The final type of result involves convergence in distribution in the space $BUC(\RR^d \times [0,T])$. Here, the paths $W_h$ are taken to be appropriately scaled simple random walks. More precisely, for some probability space $(\mcl A, \mcl G, \mbb P)$,
\begin{equation} \label{A:simplerandomwalks}
\left\{
\begin{split}
	&\lambda := \frac{(\rho_h)^{3/4}}{h} \le \frac{\theta}{\nor{H'}{\oo}}, \quad M_h := \lfloor (\rho_h)^{-1/2} \rfloor, \quad \mcl P_h := \{t_n\}_{n =0}^N = \left\{ n \rho_h \wedge T \right\}_{n \in \NN_0}, \\[1.2mm]
	&\{\xi_n\}_{n=1}^\oo: \mcl A \to \{-1, 1\} \text{ are independent,} \\[1.2mm]
	&\mbb P(\xi_n = 1) = \mbb P(\xi_n = -1) = \frac{1}{2}, \quad W(0) = 0, \quad\text{and} \\[1.2mm]
	&W_h(t) := W_h(kM_h \rho_h) + \frac{ \xi_k }{\sqrt{M_h \rho_h} }(t - kM_h \rho_h) \quad \text{for } k \in \NN_0, \; t \in [kM_h \rho_h , (k+1)M_h \rho_h).
\end{split}
\right.
\end{equation}

\begin{theorem} \label{T:introLFdistributionBM}
	If $u_h$ is constructed using \eqref{E:introschemedef} and \eqref{E:introLFfirstorder} with $W_h$ and $\mcl P_h$ as in \eqref{A:simplerandomwalks}, and $u$ is the solution of \eqref{E:simpleeqintrofirstorder} with $W$ equal to a Brownian motion, then, as $h \to 0$, $u_h$ converges to $u$ in distribution.
\end{theorem}

\subsection{Background for the study of \eqref{E:eq}}

When $W$ is continuously differentiable, or of bounded variation, the symbol $dW^i$ in equation \eqref{E:eq} stands for the time derivative $\frac{d}{dt} W^i(t) = \dot W^i(t)$ and ``$\circ$'' denotes multiplication. As already noted, the classical viscosity theory applies in this context. 

The problem becomes more complicated when $W$ is merely continuous, and therefore, possibly nowhere differentiable or of infinite variation. In many examples of interest, $W$ is the sample path of a stochastic process, such as Brownian motion, and then the symbol ``$\circ$'' is regarded as the Stratonovich differential. More generally, $W$ may be a geometric rough path, a specific instance being a Brownian motion enhanced with its Stratonovich iterated integrals. 

The notion of pathwise viscosity solutions for equations like \eqref{E:eq} was developed by Lions and Souganidis, first for Hamiltonians depending smoothly on the gradient $Du$ \cite{LSfirst}, and later for nonsmooth Hamiltonians \cite{LSnonsmooth}. The comparison principle was proved in \cite{LSsecondorder}, and equations with Hamiltonians depending nonlinearly on $u$ were considered in \cite{LSsemilinear}. The theory has since been extended to treat Hamiltonians with spatial dependence, as by Friz, Gassiat, Lions, and Souganidis \cite{FGLS}, or by the author \cite{Se}; these papers use techniques developed by Lions and Souganidis for more general settings that appear in forthcoming works \cite{LSbook}. An alternative existence result relying on Perron's method can be found in the work of the author \cite{Seperron}. Many more details and results are summarized in the notes of Souganidis \cite{Snotes}. 

The setting in which $H$ depends linearly on the gradient has been explored from the point of view of rough path theory by many authors, including, but not limited to, Caruana, Friz, and Oberhauser \cite{CFO} and Gubinelli, Tindel, and Torrecilla \cite{GTT}. The semilinear problem was also studied by Buckdahn and Ma \cite{BM1, BM2} using the pathwise control interpretation. 

It is of particular interest to have a way to analyze \eqref{E:eq} when $H$ is nonlinear and not necessarily $C^1$, because of the application, via the level set method, to the theory of the propagation of fronts with a stochastically perturbed normal velocity. For example, if, for $t > 0$, $\Gamma_t \subset \RR^d$ is a smooth, $(d-1)$-dimensional surface moving with normal velocity
\begin{equation}\label{E:stochmeancurvature}
	V = - \kappa + \alpha\; dW,
\end{equation}
where $\kappa$ is the mean curvature of the surface, $\alpha \in \RR$ is a constant, and $dW$ is white noise in time, and if $\Gamma_t$ is the $0$-level set of some function $u(\cdot,t)$, that is, $\Gamma_t = \{x \in \RR^d : u(x,t) = 0\}$, then, formally, $u$ solves the equation
\begin{equation} \label{E:meancurvature}
	du = \pars{ \Delta u - \left\langle D^2 u \frac{Du}{\abs{ Du}}, \frac{Du}{\abs{ Du}} \right\rangle } \;dt + \alpha |Du| \circ dW \quad \text{in } \RR^d \times (0,T].
\end{equation}
This is a special case of \eqref{E:eq} for which $F$ is singular. The stochastic viscosity interpretation of \eqref{E:meancurvature} has been used by Souganidis and Yip \cite{SY} to exhibit stochastic selection principles for some examples of nonuniqueness in mean curvature flow, and by Lions and Souganidis \cite{LSbook} to establish a sharp interface limit for the Allen-Cahn equation perturbed with an additive, mild approximation of time-white noise. For the latter problem, it was proved that, for some $\alpha \in \RR$, the limiting front has a normal velocity as in \eqref{E:stochmeancurvature}.

As far as we know, the results in this paper on approximation schemes for stochastic viscosity solutions are the first of their kind. We are also aware of a work by Hoel, Karlsen, Risebro, and Storr{\o}sten \cite{HKRS} using numerical methods to study a related class of equations, namely, stochastic scalar conservation laws.

\subsection{Organization of the paper}

Section \ref{S:discussion} begins with a discussion of the theory of monotone approximation schemes in the classical viscosity setting, as well as the difficulties faced for pathwise equations. In Section \ref{S:definition}, we recall some definitions and results from the theory of pathwise (stochastic) viscosity solutions. Some of the material may be found in \cite{LSfirst}, \cite{LSnonsmooth}, or \cite{Snotes}, while other facts, whose proofs are given here, are developed by Lions and Souganidis in a forthcoming work \cite{LSbook}.

In Section \ref{S:scheme}, we make the notion of the scheme operator $S_h$ more precise, and use the method of half-relaxed limits to prove that, for an appropriate family of partitions $\{ \mcl P_h\}_{h > 0}$ and paths $\{W_h\}_{h > 0}$ as in \eqref{A:introapproximators}, if $u_h$ is defined by \eqref{E:introapproxsolution} and \eqref{E:introschemedef}, then $u_h$ converges locally uniformly to the solution of \eqref{E:eq}. Various examples are presented to which the general convergence result may be applied. 

Section \ref{S:pathwise} lays the framework for the quantitative analysis of schemes for stochastic Hamilton-Jacobi equations by proving a generalization of the pathwise estimate in Theorem \ref{T:intropathwise}. This result is then used in Section \ref{S:regularrates} to obtain explicit rates of convergence, such as those stated in Theorems \ref{T:introLFresultfirstorder} and \ref{T:introLFpathwiseBM}, as well as the result on convergence in distribution as in Theorem \ref{T:introLFdistributionBM}.

\subsection{Notation}

Throughout most of the proofs in this paper, the symbol $C$ will stand for a generic constant that may change from line to line, and whose dependence will be specified or made clear from context. $C^{0,1}(\RR^d)$ is the space of Lipschitz continuous functions, and $\nor{Du}{\oo}$ is the Lipschitz constant for a function $u \in C^{0,1}(\RR^d)$. For $\alpha \in (0,1)$, $C^\alpha([0,T])$ denotes the space of $\alpha$-H\"older continuous paths on $[0,T]$, and $[W]_{\alpha,T}$ is defined to be the H\"older seminorm of $W \in C^\alpha([0,T])$. Given a continuous path $W$ and $s,t \in [0,T]$, the maximum oscillation of $W$ between $s$ and $t$ is denoted by
\begin{gather*}
	\mathrm{osc}(W,s,t) := \max_{r_1,r_2 \in I_{s,t}} \abs{ W(r_1) - W(r_2)} = \max_{I_{s,t}} W - \min_{I_{s,t}} W, \\
	\text{where} \quad I_{s,t} := \left[ \min(s,t), \max(s,t)\right].
\end{gather*}
The spaces of upper- and lower-semicontinuous functions on $\RR^d \times [0,T]$ are respectively $USC(\RR^d \times [0,T])$ and $LSC(\RR^d \times [0,T])$, and $(B)UC(U)$ is the space of (bounded) uniformly continuous functions on a domain $U$.

For $a \in \RR$, $\lfloor a \rfloor$ and $\lceil a \rceil$ denote respectively the largest (smallest) integer $k$ satisfying $k \le a$ ($k \ge a$). The mesh-size of a partition $\mcl P = \{ 0 = t_0 < t_1 < t_2 < \cdots < t_N = T\}$ of $[0,T]$ is defined by $\abs{\mcl P} := \max_{n = 0, 1, 2, \ldots, N-1} (t_{n+1} - t_n)$. $S^d$ is the space of symmetric $d$-by-$d$ matrices, and, for $X,Y \in S^d$, the inequality $X \le Y$ means $\xi \cdot X\xi \le \xi \cdot Y \xi$ for all $\xi \in \RR^d$. The set of positive integers is written as $\NN$, and $\NN_0 := \NN \cup \{0\}$.

\section{Monotone schemes for viscosity solutions} \label{S:discussion}

\subsection{The classical viscosity setting}
It is well-known that viscosity solutions of the nonlinear degenerate parabolic equation
\begin{equation} \label{E:deteq}
	u_t = F(D^2 u, Du ) \quad \text{in } \RR^d \times (0,T] \quad \text{and} \quad u(\cdot, 0) = u_0 \quad \text{in } \RR^d 
\end{equation}
satisfy a comparison principle. That is, if $u$ and $v$ are respectively a sub- and super-solution of \eqref{E:deteq}, then, for all $t \in [0,T]$,
\begin{equation}
	\sup_{ x \in \RR^d} \pars{ u(x,t) - v(x,t)} \le \sup_{x \in \RR^d} \pars{ u(x,0) - v(x,0)}. \label{E:introcomparison}
\end{equation}
In particular, if $u(\cdot,0) \le v(\cdot,0)$, then $u(\cdot,t) \le v(\cdot,t)$ for all future times $t > 0$. 

Moreover, \eqref{E:deteq} is stable under local uniform convergence. That is, if, for $n \ge 0$, $u_{0,n}, u_0 \in BUC(\RR^d)$, $F_n, F \in C(\RR^d)$, $u_n \in BUC(\RR^d \times [0,T])$ solves
\begin{equation}
	u_{n,t} = F_n(D^2u_n, Du_n) \quad \text{in } \RR^d \times (0,T] \quad \text{and} \quad u_n(\cdot,0) = u_{0,n} \quad \text{in } \RR^d, \label{E:nequation}
\end{equation}
and, as $n \to \oo$,
\begin{equation}
	u_{0,n} \to u_0 \quad \text{and} \quad F_n \to F \quad \text{locally uniformly}, \label{E:initialconv}
\end{equation}
then, as $n \to \oo$, $u_n$ converges locally uniformly to $u$, the viscosity solution of \eqref{E:deteq}.

These and other properties can be summarized in terms of the solution operators for \eqref{E:deteq}, which are, for $t \ge 0$, the maps $S(t): BUC(\RR^d) \to BUC(\RR^d)$ for which the solution $u$ of \eqref{E:deteq} is given by $u(x,t) = S(t)u_0(x)$. For all $s,t \ge 0$, $\phi, \psi \in BUC(\RR^d)$, and $k \in \RR$, these satisfy
\begin{equation}\label{E:solutionoperatorprops}
	\left\{
	\begin{split}
		(a) & \quad S(0)\phi = \phi,\\[1.2mm]
		(b) & \quad S(t+s) = S(t)S(s), \\[1.2mm]
		(c) & \quad S(t)(\phi + k) = S(t)\phi + k, \quad \text{and}\\[1.2mm]
		(d) & \quad \sup_{\RR^d} \pars{ S(t)\phi - S(t)\psi} \le \sup_{\RR^d} \pars{ \phi - \psi}.
	\end{split} 
	\right.
\end{equation}
Property \eqref{E:solutionoperatorprops}(c) implies that \eqref{E:solutionoperatorprops}(d) is equivalent to the monotonicity of $S(t)$. That is, if $\phi \le \psi$, then $S(t)\phi \le S(t)\psi$ for all $t \ge 0$. 

The stability property above can be rephrased as saying that, if \eqref{E:initialconv} holds and if $S_n(t): BUC(\RR^d) \to BUC(\RR^d)$ is the family of solution operators corresponding to \eqref{E:nequation}, then, as $n \to \oo$,  $S_n(t)u_{0,n}(x) \to S(t)u_0(x)$ locally uniformly. The philosophy behind the creation of approximation schemes is to generalize this result, by constructing, for $h > 0$ and $\rho > 0$, suitable operators $S_h(\rho): BUC(\RR^d) \to BUC(\RR^d)$ that satisfy properties similar to those in \eqref{E:solutionoperatorprops}. In particular, for all $\phi \in BUC(\RR^d)$ and $k \in \RR$, and for some increasing function $h \mapsto \rho_h$ satisfying $\lim_{h \to 0} \rho_h = 0$,
\begin{equation}\label{E:schemeoperatorprops}
	\left\{
	\begin{split}
		(a) & \quad S_h(t)(\phi + k) = S_h(t)\phi + k \quad \text{for all } h,t > 0, \\[1.2mm]
		(b) & \quad \sup_{\RR^d} \pars{ S_h(\rho)\phi - S_h(\rho)\psi} \le \sup_{\RR^d} \pars{ \phi - \psi} \quad \text{whenever } h > 0 \text{ and } 0 < \rho \le \rho_h, \text{ and} \\[1.2mm]
		(c) & \quad \lim_{h \to 0} \sup_{0 < \rho \le \rho_h} \abs{  \frac{ S_h(\rho)\phi - \phi}{\rho} - F(D^2 \phi, D\phi)} = 0 \quad \text{for all $\phi \in C^2(\RR^d)$.}
	\end{split} 
	\right.
\end{equation}
Given a partition $\mcl P_h$ satisfying $\abs{ \mcl P_h} \le \rho_h$, the approximate solution $u_h: BUC(\RR^d \times [0,T])$ is assembled by first setting $u_h(\cdot,0) := u_0$ and then iteratively defining
\begin{equation} \label{E:classicalapproxsolution}
	u_h(\cdot,t) := S_h(t - t_n)u_h(\cdot,t_n) \quad \text{for } n = 0, 1, 2, \ldots, N-1 \text{ and } t \in (t_n, t_{n+1}].
\end{equation}

One example of particular interest is the class of finite difference approximations, for which $S_h(\rho)u$ depends on the function $u$ only through its values on the discrete lattice $h\ZZ^d$. A major consideration for such schemes is to establish a relationship between the resolutions of the discrete grids in time and space, that is, to choose the map $h \mapsto \rho_h$ in such a way that the properties in \eqref{E:schemeoperatorprops} can be attained. Such a relationship is known as a Courant-Friedrichs-Lewy (CFL) condition \cite{CFL}, and various examples will be studied throughout the paper.

As is well-known, solutions of \eqref{E:eq} are generally not $C^2$ on all of $\RR^d \times [0,T]$, even if $F$, $H$, and $u_0$ are all smooth, and so
\eqref{E:schemeoperatorprops}(c) alone is not enough to prove the convergence of $u_h$ to $u$ as $h \to 0$. It is here that the monotonicity of $S_h(\rho)$, which is implied by \eqref{E:schemeoperatorprops}(a) and (b), is vital, since it allows the scheme operator to be applied to the smooth test functions coming from the definition of viscosity solutions.

A finite difference scheme operator $S_h$, in its simplest form, when $d = 1$ (the last assumption here made only to simplify the presentation), is given, for some $F_h \in C^{0,1}(\RR \times \RR \times \RR)$, by
\begin{equation} \label{E:introexplicitscheme}
	S_h(\rho)u(x) := u(x) + \rho F_h \pars{ \frac{u(x + h) + u(x - h) - 2u(x)}{h^2} ,\frac{u(x+ h) - u(x)}{h}, \frac{u(x) - u(x - h)}{h}}.
\end{equation}
The scheme \eqref{E:introexplicitscheme} automatically satisfies \eqref{E:schemeoperatorprops}(a), while \eqref{E:schemeoperatorprops}(b) holds if the function
\[
	(u, u_-, u_+) \mapsto u + \rho F_h \pars{ \frac{u_+ + u_- - 2u}{h^2} ,\frac{u_+ - u}{h}, \frac{u - u_-}{h}}
\]
is nondecreasing in each argument when $0 < \rho \le \rho_h$, which, in turn, calls for
\begin{equation}
	\rho_h := \lambda h^2 \label{E:introsecondorderCFL}
\end{equation}
for some sufficiently small constant $\lambda > 0$. In the case of first-order equations, that is, for the equation
\begin{equation}
	u_t = H(Du) \quad \text{in } \RR^d \times (0,T] \quad \text{and} \quad u(\cdot, 0) = u_0 \quad \text{on } \RR^d, \label{E:firstorderdeteq}
\end{equation}
the CFL condition becomes
\begin{equation}
	\rho_h = \lambda h. \label{E:introfirstorderCFL}
\end{equation}

The function $F_h$ is related to $F$ through a consistency requirement, which here means that, for all $X \in \RR$ and $p \in \RR$,
\begin{equation}
	\lim_{h \to 0} F_h(X,p,p) = F(X,p) \quad \text{and} \quad \sup_{h > 0} \nor{DF_h}{\oo} < \oo. \label{E:intronHconsistency}
\end{equation}
Property \eqref{E:schemeoperatorprops}(c) can then be readily verified by using Taylor approximations to estimate the finite differences of functions $\phi \in C^2(\RR^d)$. 

An instructive example in the first-order setting is the following analogue of the Lax-Friedrichs scheme for scalar conservation laws. Let $\epsilon_h > 0$ and define, for $x \in \RR$,
\begin{equation} \label{E:introFLscheme}
	S_h(\rho)u(x) := u(x) + \rho \left\{ H \pars{ \frac{u(x+h) - u(x-h)}{2h}} + \epsilon_h \pars{ \frac{u(x+h) + u(x-h) - 2u(x)}{h^2}} \right\}.
\end{equation}
Here, $H_h$ is given by
\[
	H_h(p,q) = H\pars{ \frac{p+q}{2}} + \frac{\epsilon_h}{h} (p-q).
\]
The final term in \eqref{E:introFLscheme} is a discrete analogue of the method of vanishing viscosity, and is used here to inject monotonicity into the scheme. Indeed, if, for some fixed $\theta > 0$ and $\lambda > 0$, the small parameter $\epsilon_h$ is defined by
\begin{equation} \label{E:epsilonrelation}
	\epsilon_h := \frac{\theta h}{2\lambda},
\end{equation}
then \eqref{E:schemeoperatorprops}(b) is satisfied as long as \eqref{E:introfirstorderCFL} holds with $\theta \le 1$ and $\lambda \le \frac{\theta}{\nor{H'}{\oo}}$.

In \cite{CL}, Crandall and Lions found explicit error estimates for this and and other explicit finite difference schemes for homogenous Hamilton-Jacobi equations. More precisely, it was proved for the above example that there exists a constant $C > 0$ depending only on $\nor{DH}{\oo}$, $\nor{Du_0}{\oo}$, and $\lambda$ such that, if $u_h$ is defined as in \eqref{E:classicalapproxsolution} and \eqref{E:introFLscheme}, and if $u$ solves \eqref{E:firstorderdeteq}, then
\[
	\sup_{(x,t) \in \RR^d \times [0,T]} \abs{ u_h(x,t) - u(x,t) } \le C(1+T)h^{1/2}.
\]
This same rate was later established by Souganidis \cite{S} for both explicit and implicit finite difference schemes for equations with Lipschitz spatial and time dependence, and the same method was applied to study other approximations such as max-min representations and Trotter-Kato product formulas \cite{Sproduct}. 

Barles and Souganidis \cite{BS} considered schemes for second order equations, using a shorter, qualitative proof of convergence relying on the method of half-relaxed limits. Kuo and Trudinger \cite{KT1, KT2} also investigated such schemes in great detail and constructed several examples. The question of estimating the rates of convergence for such approximations of second order equations was analyzed from many points of view. Barles and Jakobsen \cite{BJ1, BJ2, BJ3} achieved algebraic convergence rates for stochastic control problems, taking advantage of the fact that $F$ is convex in that setting. Jakobsen \cite{J1, J2} and Krylov \cite{K} also established rates of convergence for nonconvex problems under some restrictions on $F$. If $F$ is uniformly elliptic, then rates of convergence can be found under very general assumptions using techniques from the regularity theory for fully nonlinear, uniformly elliptic equations, as exhibited by Caffarelli and Souganidis \cite{CS}, and later by Turanova \cite{T} for inhomogenous equations.

\subsection{Difficulties in the pathwise setting}

The lack of regularity for $W$ complicates the task of constructing scheme operators for \eqref{E:eq} that are both monotone and consistent. 

Consider, for example, modifying the Lax-Friedrichs scheme \eqref{E:introFLscheme} for the stochastic Hamilton Jacobi equation
\begin{equation} \label{E:HJeq}
	du = H(u_x) \circ dW \quad \text{in } \RR \times (0,T] \quad \text{and} \quad u(\cdot,0) = u_0 \quad \text{in } \RR.
\end{equation}
If $W$ is sufficiently regular, then it is reasonable to define a time-inhomogenous scheme operator by
\begin{equation}
	\begin{split}
	S_h(t, s)u(x) &:= u(x) + H \pars{ \frac{u(x+h) - u(x-h)}{2h} }(W(t) - W(s))\\[1.2mm]
	& + \epsilon_h \pars{ \frac{ u(x+h) + u(x-h) - 2u(x)}{h^2} }(t-s).
	\end{split} \label{E:pathwiseguess}
\end{equation}
Proceeding as in the previous subsection, a simple calculation reveals that $S_h(t,s)$ is monotone for $0 \le t-s \le \rho_h$, if $\rho_h$ and $\epsilon_h$ are such that, for some $\theta \in (0, 1]$,
\[
	\epsilon_h := \frac{\theta h^2}{2(t-s)}
\] 
and
\begin{equation}\label{E:intropathCFL}
	\lambda := \max_{|t-s| \le \rho_h} \frac{ \mathrm{osc}(W, s,t) }{h} \le \lambda_0 := \frac{\theta}{\nor{H'}{\oo}}.
\end{equation}
On the other hand, spatially smooth solutions $\Phi$ of \eqref{E:HJeq} have the expansion, for any $s,t \in [0,T]$ with $|s-t|$ sufficiently small,
\begin{equation}\label{E:introsmoothexpansion}
	\begin{split}
	\Phi(x,t) = \Phi(x,s) &+ H(\Phi_x(x,s))(W(t)-W(s)) \\[1.2mm]
	&+ H'(\Phi_x(x,s))^2 \Phi_{xx}(x,s) (W(t)-W(s))^2 + O(\abs{ W(t) - W(s)}^3),
	\end{split}
\end{equation}
so that, if $0 \le t-s \le \rho_h$, for some $C > 0$ depending only on $H$,
\begin{equation}\label{E:pathconsistencyconsequence}
	\begin{split}
	\sup_{\RR} \abs{ S_h(t , s)\Phi(\cdot,s) -\Phi(\cdot,t)} &\le C \sup_{r \in [s,t]} \nor{D^2\Phi(\cdot,r)}{\oo} \pars{ \abs{W(t) - W(s)}^2 + h^2} \\
	&\le C \sup_{r \in [s,t]} \nor{D^2\Phi(\cdot,r)}{\oo} (1 + \lambda_0^2)h^2.
	\end{split}
\end{equation}
Therefore, in order for the scheme to have a chance of converging, $\rho_h$ should satisfy
\begin{equation}\label{E:introlimits}
	\lim_{h \to 0} \frac{h^2}{\rho_h} = 0.
\end{equation}
Both \eqref{E:intropathCFL} and \eqref{E:introlimits} can be achieved when $W$ is continuously differentiable, or merely Lipschitz, by setting
\[
	\rho_h := \frac{\lambda h}{ \nor{\dot W}{\oo}}.
\]
More generally, if $W$ has Young-H\"older regularity, that is, $W \in C^\alpha([0,T])$ with $\alpha > \frac{1}{2}$, and if
\begin{equation} \label{E:alphaCFL}
	(\rho_h)^\alpha := \frac{\lambda h}{ [W]_{\alpha,T}},
\end{equation}
then both \eqref{E:intropathCFL} and \eqref{E:introlimits} are satisfied, since 
\[
	\frac{h^2}{\rho_h} = \pars{ \frac{[W]_{\alpha,T} h^{2\alpha-1}}{\lambda} }^{1/\alpha} \xrightarrow{h\to 0} 0.
\]
However, this approach fails as soon as the quadratic variation
\[
	\mcl Q([0,T],W) := \lim_{|\mcl P| \to 0} \sum_{n=0}^{N-1} \abs{ W(t_{n+1}) - W(t_n)}^2
\]
is non-zero, as \eqref{E:intropathCFL} and \eqref{E:introlimits} together imply that $\mcl Q([0,T], W) = 0$. This rules out, for instance, the case where $W$ is the sample path of a Brownian motion, for which $\mcl Q([0,T],W) = T$ with probability one.  

Motivated by the theory of rough differential equations, it is natural to explore whether the scheme operator \eqref{E:pathwiseguess} can be altered in some way to refine the estimate in \eqref{E:pathconsistencyconsequence}, potentially allowing \eqref{E:introlimits} to be relaxed and $\rho_h$ to converge more quickly to zero as $h \to 0^+$. More precisely, the next term in the expansion \eqref{E:introsmoothexpansion} suggests that, for $W \in C^\alpha([0,T], \RR)$ with $\alpha > \frac13$ (or more generally, $W$ with $p$-variation for $p < 3$), one should define
\begin{equation} \label{E:higherorderguess}
	\begin{split}
	S_h(t, s)u(x) &:= u(x) + H \pars{ \frac{u(x+h) - u(x-h)}{2h} }(W(t) - W(s))\\
	& + \frac{1}{2} H'\pars{ \frac{u(x+h) - u(x-h)}{2h}}^2 \pars{ \frac{u(x+h) + u(x-h) - 2u(x) }{h^2}} \pars{ W(t) - W(s)}^2 \\
	& + \frac{\theta}{2} \pars{  u(x+h) + u(x-h) - 2u(x) }.
	\end{split}
\end{equation}
As can easily be checked, \eqref{E:higherorderguess} is monotone as long as \eqref{E:intropathCFL} holds,
\[
	\nor{Du}{\oo} \le L, \quad \theta + \nor{H'}{\oo} \lambda^2 \le 1, \quad \text{and} \quad \lambda \le \frac{\theta}{\nor{H'}{\oo} \pars{ 1 + 2 L \nor{H''}{\oo}}}.
\]
On the other hand, the error in \eqref{E:pathconsistencyconsequence} would then be of order $h^2 + \abs{ W(t) - W(s)}^3$, which again leads to \eqref{E:introlimits}. This seems to indicate that we should also incorporate higher order corrections in \eqref{E:higherorderguess} to improve the order of the error in the $h$ variable. However, this will disrupt the monotonicity of the scheme in general. This is due to the fact that such terms involve the discrete second derivative of $u$, and, thus, will counter the effect of the term
\[
	\frac{\theta}{2} \pars{ u(x+h) + u(x-h) - 2u(x)},
\]
which is included precisely for the purpose of creating monotonicity.

For this reason, we develop a more effective strategy that works for any continuous path. Namely, rather than modifying the scheme itself, we regularize the path $W$. If $\{W_h\}_{h > 0}$ is a family of smooth paths converging uniformly, as $h \to 0$, to $W$, then $\mcl Q(W_h, [0,T]) = 0$ for each fixed $h > 0$, and therefore, $W_h$ and $\rho_h$ can be chosen so that \eqref{E:intropathCFL} and \eqref{E:introlimits} hold for $W_h$ rather than $W$. Various methods for implementing this procedure, both qualitative and quantitative, are explored throughout the paper.

\section{The definition of pathwise viscosity solutions} \label{S:definition}

\subsection{Assumptions on the nonlinearities}

The nonlinear function $F: S^d \times \RR^d \to \RR$ is assumed to be Lipschitz and degenerate elliptic; that is,
\begin{equation} \label{A:F}
	\left\{
	\begin{split}
		&F \in C^{0,1}(S^d \times \RR^d) \quad \text{and}\\[1.2mm]
		&F(X,p) \le F(Y,p) \quad \text{whenever} \quad p \in \RR^d \quad \text{and} \quad X \le Y.
	\end{split}
	\right.
\end{equation}
The results of this paper may be extended to the case where $F$ has additional dependence on $u$, $x$, or $t$, in which case $F$ requires additional structure in order for the comparison principle to hold. To simplify the presentation, we take $F$ as in \eqref{A:F}. One consequence is that the solution operator for \eqref{E:eq} is invariant under translations in both the independent and dependent variables.

In order for \eqref{E:eq} to be well-posed for all continuous $W$ and uniformly continuous $u_0$, the Hamiltonians need to be more regular than what is required in the classical viscosity theory. As explained in \cite{LSnonsmooth} and \cite{Snotes}, it is necessary to assume that
\begin{equation}
	H^i = H^i_1 - H^i_2 \text{ for convex $H^i_1,H^i_2 : \RR^d \to \RR$ with $H^i_1, H^i_2 \ge 0$.} \label{A:Hdiffconvex}
\end{equation}
The non-negativity is imposed here only to simplify some arguments in what follows, and the setting can be reduced to the general case by transforming the equation appropriately. 

Letting $H$ depend additionally on $u$ or $x$ makes the question of well-posedness for \eqref{E:eq} highly nontrivial. Indeed, there is no pathwise theory for equations of the form
\[
	du = F(D^2 u, Du)\;dt + \sum_{i=1}^m H^i(Du, u, x) \circ dW^i,
\]
except for some special cases, for instance, if the dependence of $H$ on $Du$ is linear. 

Under certain assumptions, \eqref{E:eq} is well-posed for $H$ depending nonlinearly on both $Du$ and $x$. In this case, the lack of uniform regularity estimates for the solutions becomes an obstacle in the construction of schemes for \eqref{E:eq}. These issues will be the subject of a future work.

The homogeneity of $H$ in space allows us to forego difficult questions about regularity, because the spatial modulus of continuity for the solution of \eqref{E:eq} is retained for all time. In particular, throughout much of the paper, the initial condition $u_0$ is fixed and satisfies
\begin{equation} \label{A:initialcondition}	
	u_0 \in C^{0,1}(\RR^d) \quad \text{and} \quad \nor{Du_0}{\oo} \le L,
\end{equation}
and therefore, $H(p)$ may be redefined for $|p| > L$ without affecting the solution. Since \eqref{A:Hdiffconvex} implies that $H$ is locally Lipschitz, we may then assume that
\begin{equation}
	\text{for some } C = C_L > 0, \quad \nor{DH}{\oo} = C_L < \oo. \label{A:Hlingrowth}
\end{equation}
Note that \eqref{A:Hlingrowth} implies that $H$ grows at most linearly as $|p| \to +\oo$. 

In some parts of the paper, to allow for a more flexible solution theory, especially when we discuss schemes for second order equations, the hypothesis \eqref{A:Hdiffconvex} is replaced with the stronger assumption
\begin{equation}\label{A:HC2}
	H \in C^k(\RR^d, \RR^m) \quad \text{for some } k = 2, 3, \ldots.
\end{equation}
If $H$ satisfies \eqref{A:HC2}, and therefore $H \in C^{1,1}(\RR^d,\RR^m)$, then \eqref{A:Hdiffconvex} is satisfied on every bounded set of gradients.

\subsection{Smooth solutions of the Hamilton-Jacobi part of \eqref{E:eq}}

The definition of pathwise viscosity solutions relies on the existence of local-in-time, smooth-in-space solutions of the Hamilton-Jacobi part of equation \eqref{E:eq}. More precisely, for $t_0 \in [0,T]$ and $\phi \in C^{1,1}(\RR^d)$, the goal is to find an open interval $I \subset [0,T]$ containing $t_0$ and a solution $\Phi \in C(I, C^{1,1}(\RR^d))$ of
\begin{equation}\label{E:HJpart}
	d\Phi = \sum_{i=1}^m H^i(D\Phi) \circ dW^i \quad \text{in } \RR^d \times I \quad \text{and} \quad \Phi(\cdot,t_0) = \phi \quad \text{in } \RR^d.
\end{equation}
Such solutions are defined through a density argument, that is, the solution operator for \eqref{E:HJpart} for smooth paths extends continuously to continuous paths. This is justified with the computations below, and is consistent with the cases where $W$ is a Brownian motion or, more generally, a geometric rough path.

When $H$ and $\phi$ are smooth, the construction of such solutions can be accomplished for any smooth $\phi$ by inverting the characteristics associated to \eqref{E:HJpart}. Because $H$ is independent of $x$, this amounts to inverting the map
\begin{equation}
	x \mapsto X(x,t) := x - \sum_{i=1}^m DH^i(D\phi(x)) \pars{ W^i(t) - W^i(t_0)}. \label{E:characteristic}
\end{equation}
The continuity of $W$ implies that there exists an interval $I \ni t_0$ such that
\[
	\sup_{t \in I} \abs{ W(t) - W(t_0)} < \frac{1}{\nor{D^2 H}{\oo} \nor{D^2\phi}{\oo}},
\]
whence \eqref{E:characteristic} is invertible for all $t \in I$. The solution is then given by $\Phi(x,t) := Z(X\nv(x,t),t)$, where
\begin{equation}
	Z(x,t) := \phi(x) + \sum_{i=1}^m \pars{ H^i(D\phi(x)) - D\phi(x) \cdot DH^i(D\phi(x)) } \pars{ W^i(t) - W^i(t_0)}. \label{E:Zcharacteristic}
\end{equation}
This can be confirmed with a simple calculation when $W$ is smooth. For general continuous paths, the formula holds by a density argument, since all expressions depend only on the values of $W$, and not on its derivatives. 

Notice also that the regularity of $\Phi$ improves with that of $H$ and $\phi$. Indeed, differentiating \eqref{E:Zcharacteristic} leads to the relation $D\Phi(X(x,t),t) = D\phi(x,t)$, and therefore, in view of \eqref{E:characteristic}, the solution $\Phi$ belongs to $C(I, C^k(\RR^d))$ as long as $H \in C^k(\RR^d, \RR^m)$ and $\phi \in C^k(\RR^d)$ for some $k = 2, 3, \ldots$. Furthermore, if $D^j \phi$ is bounded for some $j = 0, 1, 2, \ldots, k$, then, shrinking $I$ if necessary,
\[
	\sup_{t \in I} \nor{D^j \Phi(\cdot,t)}{\oo} < \oo.
\]

This strategy breaks down when $H$ is only assumed to satisfy \eqref{A:Hdiffconvex}, since such Hamiltonians are not even continuously differentiable in general. In this case, only very particular smooth solutions of \eqref{E:HJpart} can be constructed. Assume $\eta: \RR^d \to \RR$ is strictly convex, and, for $\delta > 0$, define
\begin{equation}
	\Phi(x,t) := \sup_{p \in \RR^d} \left\{ p \cdot x - \eta(p) - \delta \sum_{i=1}^m \pars{ H^i_1(p) + H^i_2(p)} + \sum_{i=1}^m H^i(p) (W^i(t) - W^i(t_0)) \right\}. \label{E:smoothsolution}
\end{equation}

\begin{lemma} \label{L:smoothsolutions}
	Let $H$ satisfy \eqref{A:Hdiffconvex} and \eqref{A:Hlingrowth}. If the open interval $I \ni t_0$ is such that
	\[
		\sup_{t \in I} \max_{i=1,2,\ldots,m} \abs{ W^i(t) - W^i(t_0)} < \delta,
	\]
	then the function $\Phi$ defined by \eqref{E:smoothsolution} belongs to $C(I, C^{1,1}(\RR^d))$, and is a solution of \eqref{E:HJpart} with
	\begin{equation} \label{E:initialphi}
		\Phi(\cdot, t_0) = \phi(x) := \sup_{p \in \RR^d} \left\{ p \cdot x - \eta(p) - \delta \sum_{i=1}^m \pars{ H^i_1(p) + H^i_2(p)} \right\}.
	\end{equation}
\end{lemma}

\begin{proof}
	For all $x \in \RR^d$ and $t \in I$, the function
	\[
		p \mapsto \eta(p) + \delta \sum_{i=1}^m \pars{ H^i_1(p) + H^i_2(p)} - \sum_{i=1}^m H^i(p) \pars{ W^i(t) - W^i(t_0)} - p \cdot x
	\]
	is strictly convex, and therefore attains a unique global minimum. The smoothness of $\Phi$ in $x$ then follows from the implicit function theorem.
	
	Now, for $t \in \RR$, let $S^i(t) : UC(\RR^d) \to UC(\RR^d)$ be the solution operator for the equation $u_t = H^i(Du)$. If $\psi \in UC(\RR^d)$ is convex, then the Hopf formula gives
	\[
		S^i(t)\psi(x) = \sup_{p \in \RR^d} \left\{ p \cdot x - \psi^*(p) + tH^i(p) \right\},
	\]
	and so \eqref{E:smoothsolution} can be rewritten as
	\[
		\Phi(x,t) = \prod_{i=1}^m S^i(W^i(t) - W^i(t_0))\phi(x)
	\]
	with $\phi$ as in \eqref{E:initialphi}. If $W$ is smooth, then the fact that $\Phi$ is a solution of \eqref{E:HJpart} is justified by the regularity of $\Phi$ and a simple calculation. The result holds for continuous $W$ by a density argument.
\end{proof}

As in the classical viscosity theory, many quantitative arguments involve doubling variables, and it is therefore important to have a smooth solution of \eqref{E:HJpart} that behaves like the penalizing ``distance function''
\begin{equation}
	(x,y) \mapsto \frac{|x-y|^2}{2\delta}. \label{E:classicaldistance}
\end{equation}
In the present setting, this is accompished with a function $\Phi_\delta: \RR^d \times [0,T]^2 \times C([0,T], \RR^m) \to \RR$ that is equal to a particular choice of \eqref{E:smoothsolution} near the diagonal $\{ (t,t) \in [0,T]^2\}$, and such that $\Phi_\delta(x-y,s,t;W)$ exhibits similar growth as \eqref{E:classicaldistance} when $|x-y|$ is large.

Define the neighborhood $U_\delta(W)$ by
\[
	U_\delta(W) := \left\{ (s,t) \in [0,T]^2 : \mathrm{osc}(W,s,t) < \delta \right\},
\]
let the projection $\pi_\delta(W): [0,T]^2 \to \oline{U_\delta(W)}$ be such that $\pi_\delta(W)(s,t)$ is the element $(\tilde s,\tilde t) \in \oline{U_\delta(W)}$ closest to $(s,t)$ on the line $\tilde s + \tilde t = s + t$, and set
\begin{equation}\label{E:distancefunction}
	\begin{split}
	\Phi_\delta(x,s,t; W) := 
	\begin{dcases}
		\sup_{p \in \RR^d} \left\{ p \cdot x - \frac{\delta}{2} |p|^2 - \delta \sum_{i=1}^m (H^i_1(p) + H^i_2(p)) \right. &\\
		\qquad + \left. \sum_{i=1}^m H^i(p) \pars{ W^i(s) - W^i(t)} \right\} & \text{if } (s,t) \in \oline{U_\delta(W)},\\[1.2mm]
		\Phi_\delta\pars{ x, \pi_\delta(W)(s,t); W} & \text{if } (s,t) \notin \oline{U_\delta(W)}.
	\end{dcases}
	\end{split}
\end{equation}

\begin{lemma} \label{L:distancefunction}
	Assume $H$ satisfies \eqref{A:Hdiffconvex} and \eqref{A:Hlingrowth}, and let $\Phi_\delta$ be defined as in \eqref{E:distancefunction}. For some $C = C_L > 0$ and for all $\delta > 0$ and $W \in C([0,T]; \RR^m)$, the following hold:
	\begin{enumerate}
	\item[(a)] For all $x \in \RR^d$ and $(s,t), (\tilde s, t) \in U_\delta(W)$,
		\[
			\abs{ \Phi_\delta(x,s, t; W) - \Phi_\delta(x,\tilde s, t; W) } \le C\pars{ 1 + \frac{|x|}{\delta}} \abs{ W(s) - W(\tilde s)}.
		\]
		
	\item[(b)] For all $(s,t) \in [0,T]^2$, $\Phi_\delta(\cdot, s,t; W)$ is convex and semiconcave with constant $\frac{1}{\delta}$. That is,
		\[
			0 \le D^2\Phi_\delta(x,s,t;W) \le \frac{1}{\delta} I_d \quad \text{in the sense of distributions.}
		\]
	\item[(c)] For all $x \in \RR^d$ and $s,t \in [0,T]$,
		\[
			\frac{1}{2(C+1)\delta} |x|^2 - C\delta \le \Phi_\delta(x,s,t ; W) \le \frac{1}{2\delta} |x|^2. 
		\]
	\item[(d)] For any fixed $y \in \RR^d$ and $t \in [0,T]$, the functions
		\[
			(x,s) \mapsto \Phi_\delta(x - y, s,t; W) \quad \text{and} \quad (x,s) \mapsto - \Phi_\delta(y - x, t, s ; W)
		\]
		are $C(I, C^{1,1}(\RR^d))$-solutions of \eqref{E:HJpart}, where $I := \left\{ s \in [0,T] : \mathrm{osc}(W,s,t) < \delta \right\}$.
	\end{enumerate}
\end{lemma}

Note that the local regularity given by (a) also applies to the second time variable, in view of the identity $\Phi_\delta(x,s,t; W) = \Phi_\delta(x,t,s;-W)$.

\begin{proof}[Proof of Lemma \ref{L:distancefunction}]
	To prove (a), we first show that there exists $C = C_L > 0$ such that, for any $x \in \RR^d$ and $(s,t) \in U_\delta(W)$, the unique maximum $p^*$ achieved in the definition of $\Phi_\delta$ satisfies $\delta \abs{ p^*} \le C\delta + |x|$. Indeed, if
	\[
		J(p) := p \cdot x - \frac{\delta}{2} |p|^2 - \delta \sum_{i=1}^m (H^i_1(p) + H^i_2(p)) + \sum_{i=1}^m H^i(p) \pars{ W^i(s) - W^i(t)},
	\]
	then, for any $q \in \RR^d$, \eqref{A:Hlingrowth} and the inequality $J(p^*) \ge J(p^* + q)$ imply that
	\[
		\delta p^* \cdot \frac{q}{|q|} - \frac{\delta}{2} |q| \le |x| + C\delta.
	\]
	Setting $q = t \frac{p^*}{|p^*|}$ and sending $t \to 0^+$ yields the claim. The time-regularity estimate in (a) is then immediate.
	
	As a pointwise supremum of affine functions, $\Phi_\delta$ is clearly convex, while the semiconcavity follows from elementary convex analysis and the convexity of
	\[
		p \mapsto \delta \sum_{i=1}^m (H^i_1(p) + H^i_2(p)) - \sum_{i=1}^m H^i(p) \pars{ W^i(s) - W^i(t)}.
	\]
	The estimate in (c) can be deduced from Young's inequality and the fact that, for some $C = C_L> 0$ and for all $p \in \RR^d$ and $(s,t) \in U_\delta(W)$,
	\[
		0 \le \delta \sum_{i=1}^m (H^i_1(p) + H^i_2(p)) - \sum_{i=1}^m H^i(p) \pars{ W^i(s) - W^i(t)}  \le C\delta(1 + |p|).
	\]
	
	Finally, (d) is a consequence of Lemma \ref{L:smoothsolutions}.
\end{proof}

The following definition for solutions of \eqref{E:eq} relies on the existence of solutions of \eqref{E:HJpart} that are $C^2$, and, in particular, is only valid if $H$ is at least twice continuously differentiable.

\begin{definition} \label{D:smoothH}
	A function $u \in USC(\RR^d \times [0,T])$ (resp. $u \in LSC(\RR^d \times [0,T])$) is called a pathwise viscosity sub-solution (resp. super-solution) of \eqref{E:eq} for $H$ satisfying \eqref{A:HC2} if $u(\cdot,0) \le u_0$ (resp. $u(\cdot,0) \ge u_0$) and, whenever $\psi \in C^1([0,T])$, $(x_0,t_0) \in \RR^d \times [0,T]$, $I \ni t_0$, $\Phi \in C(I, C^{2}(\RR^d))$ is a solution of \eqref{E:HJpart} in $\RR^d \times I$, and
\[
	u(x,t) - \Phi(x,t) - \psi(t)
\]
attains a local maximum (resp. minimum) at $(x_0,t_0) \in \RR^d \times I$, then
\begin{equation}\label{E:solutionineq}
	\psi'(t_0) \le F(D^2\Phi(x_0,t_0), D\Phi(x_0,t_0)) \quad \pars{ \text{resp. } \psi'(t_0) \ge F(D^2\Phi(x_0,t_0), D\Phi(x_0,t_0))}.
\end{equation} 
A solution of \eqref{E:eq} is both a sub- and super-solution.
\end{definition}

If $\Phi \in C(I, C^{1,1}(\RR^d))$ is a solution of \eqref{E:HJpart}, then it is not possible to make sense of \eqref{E:solutionineq}, since $D^2\Phi$ may not be defined at every point. The following definition is made to comply with the case when $H$ only satisfies \eqref{A:Hdiffconvex}.

\begin{definition} \label{D:nonsmoothH}
	A function $u \in USC(\RR^d \times [0,T])$ (resp. $u \in LSC(\RR^d \times [0,T])$) is called a pathwise viscosity sub-solution (resp. super-solution) of \eqref{E:eq} if $u(\cdot,0) \le u_0$ (resp. $u(\cdot,0) \ge u_0$) and, whenever $I \subset [0,T]$ and $\Phi \in C(I, C^{1,1}(\RR^d))$ is a solution of \eqref{E:HJpart} in $\RR^d \times I$, the function $v: \RR^d \times I \to \RR$ defined by
\[
	v(\xi,t) := \sup_{x \in \RR^d} \left\{ u(x,t) - \Phi(x-\xi,t) \right\} \quad \pars{ \text{resp. } v(\xi,t) := \inf_{x \in \RR^d} \left\{ u(x,t) + \Phi(x - \xi,t) \right\} }
\]
is a classical viscosity sub- (resp. super-) solution of the equation
\[
	v_t = F(D^2 v, Dv) \quad \text{in } \RR^d \times I.
\]
A solution of \eqref{E:eq} is both a sub- and super-solution.
\end{definition}

When $H$ satisfies \eqref{A:HC2}, Definition \ref{D:nonsmoothH} is equivalent to Definition \ref{D:smoothH}. In the first-order setting, that is, when $F \equiv 0$, Definition \ref{D:smoothH} may be used even if $H$ is not smooth, because it is not necessary to evaluate $D^2\Phi$ at any point.

With either definition, \eqref{E:eq} satisfies the following comparison principle, a proof for which can be found in \cite{LSsecondorder} or \cite{Snotes}: if $u \in USC(\RR^d \times [0,T])$ and $v \in LSC(\RR^d \times [0,T])$ are respectively a sub- and super-solution of \eqref{E:eq}, then, for all $t \in (0,T]$,
\begin{equation}
	\sup_{x \in \RR^d} \pars{ u(x,t) - v(x,t)} \le \sup_{x \in \RR^n} \pars{ u(x,0) - v(x,0)}.  \label{E:comparison}
\end{equation}
A variant of the proof of the comparison principle gives the following path-stability estimate \cite{Snotes}.

\begin{lemma} \label{L:LSestimate}
	Assume that $H$ satisfies \eqref{A:Hdiffconvex}. There exists $C = C_L > 0$ such that, if $u_0 \in C^{0,1}(\RR^d)$ with $\nor{Du_0}{\oo} \le L$, $W^1,W^2 \in C([0,T], \RR^m)$, and $u^1, u^2 \in C([0,T], C^{0,1}(\RR^d))$ are the solutions of \eqref{E:eq} with respectively the paths $W^1$ and $W^2$, then
	\[
		\sup_{ (x,t) \in \RR^d \times [0,T]} \abs{ u^1(x,t) - u^2(x,t) } \le C \max_{t \in [0,T] } \abs{ W^1(t) - W^2(t)}.
	\]
\end{lemma}

It can be shown that, when $W \in C^1([0,T])$, the above notions of pathwise viscosity solutions are consistent with the standard definitions from the classical viscosity theory. Furthermore, solutions of \eqref{E:eq} are stable under uniform convergence. Therefore, the estimate in Lemma \ref{L:LSestimate}, which is proved first for smooth paths, also establishes the existence of pathwise viscosity solutions. Lemma \ref{L:LSestimate} can then be seen to hold for arbitrary continuous paths via a density argument.

Although the class of test functions defined by \eqref{E:smoothsolution} is rather restrictive, it is enough to prove both the comparison principle and the stability estimate. Indeed, only the ``distance function'' $\Phi_\delta$ in \eqref{E:distancefunction} is used in both proofs. When $H \in C^2(\RR^d)$, any initial condition $\Phi(\cdot,t_0) \in C^2(\RR^d)$ with bounded second derivatives yields a solution as in \eqref{E:HJpart}. In particular, by adding quadratic functions to $\Phi(\cdot,t_0)$, it may be assumed that the test functions in Definition \ref{D:smoothH} satisfy
\[
	\lim_{|x| \to +\oo} \frac{\Phi(x,t)}{|x|} = +\oo \quad \text{or} \quad \lim_{|x| \to +\oo} \frac{\Phi(x,t)}{|x|} = -\oo\qquad \text{uniformly for $t \in I$}.
\]
Thus, as in the classical viscosity theory, the maxima and minima in Definition \ref{D:smoothH} may be assumed to be strict without loss of generality.

Finally, we remark that, if $H$ satisfies \eqref{A:HC2}, then it is enough to use functions $\Phi \in C(I, C^k(\RR^d))$ in Definition \ref{D:smoothH}. The argument is almost identical to one from the classical viscosity theory, and it uses the fact that the solution operator for \eqref{E:HJpart} is contractive. 

\section{The general convergence result and applications} \label{S:scheme}

The constructions in this paper rely on the properties of a family of scheme operators, indexed by $h > 0$, $s,t \in [0,T]$ with $s \le t$, and a path $\zeta \in C([0,T], \RR^m)$:
\[
	S_h(t,s; \zeta) : (B)UC(\RR^d) \to (B)UC(\RR^d).
\]
We assume throughout that $S_h$ commutes with translations in both the independent and dependent variables, in order to reflect the corresponding translation invariance of \eqref{E:eq}. That is,
\begin{equation}\label{A:constantcommute}
	S_h(t,s; \zeta)(u + k) = S_h(t,s; \zeta)u + k \quad \text{for all} \quad k \in \RR \quad \text{and} \quad u \in (B)UC(\RR^d),
\end{equation} 
and
\begin{equation}\label{A:translatecommute}
	S_h(t,s; \zeta)\circ \tau_v = \tau_v \circ S_h(t,s; \zeta) \quad \text{for all} \quad v \in \RR^d, \quad \text{where }\tau_v u := u(\cdot + v).
\end{equation}

For a Hamiltonian $H$ satisfying \eqref{A:HC2} and a fixed continuous path $W \in C([0,T], \RR^m)$, we consider a family of paths $\{W_h\}_{h > 0} \subset C([0,T], \RR^m)$ and a partition width $\rho_h > 0$ satisfying
\begin{equation} \label{A:approxpath}
	h \mapsto \rho_h \text{ is increasing,} \quad \lim_{h \to 0} \nor{W_h - W}{\oo} = 0 = \lim_{h \to 0} \rho_h;
\end{equation}
\begin{equation}\label{A:monotonicity}
	\text{if } u_1 \le u_2 \text{ and } s,t \in [0,T] \text{ satisfy } 0 \le t - s \le \rho_h, \quad \text{then } S_h(t,s; W_h)u_1 \le S_h(t,s; W_h)u_2; 
\end{equation}
and
\begin{equation}\label{A:consistency}
	\left\{
	\begin{split}
		&\text{if $I \subset \RR$, $\Phi_h \in C(I, C^k(\RR^d))$ is a solution of } d\Phi_h = \sum_{i=1}^m H^i(D\Phi_h)\circ dW^i_h \text{ in } \RR^d \times I,\\
		&s_h, t_h \in I, \; 0 \le t_h - s_h \le \rho_h, \; \phi \in C^k(\RR^d), \; R > 0,\;  \text{and} \; \lim_{h \to 0} \nor{\Phi_h(\cdot,s_h) - \phi}{C^k(\RR^d)} = 0, \\[1.2mm]
		&\text{then }\lim_{h \to 0} \frac{ S_h(t_h,s_h; W_h)\Phi_h(\cdot,s_h)(x) - \Phi_h(x,s_h)  }{t_h - s_h} = F(D^2\phi(x), D\phi(x))\\[1.2mm]
		&\text{uniformly for $ x \in \RR^d$ and } \max_{j = 2,3, \ldots, k} \nor{D^j \phi}{\oo} \le R.
	\end{split}
	\right.
\end{equation}
The integer $k$ in \eqref{A:consistency} corresponds to the level of regularity of $H$ in \eqref{A:HC2}. In Sections \ref{S:pathwise} and \ref{S:regularrates}, we obtain error estimates for schemes for first-order equations with Hamiltonians satisfying the weaker condition \eqref{A:Hdiffconvex}, in which case the assumptions on the scheme operator will be modified.

The scheme operator is used to build approximate solutions as follows. For a fixed path $\zeta \in C([0,T]; \RR^m)$, partition $\mcl P = \left\{0 = t_0 < t_1 < \cdots < t_N = T \right\}$ of $[0,T]$, and initial datum $u_0 \in BUC(\RR^d)$, define
\begin{equation} \label{E:approxsolution}
	\begin{dcases}
		v_h(\cdot,0; \zeta, \mcl P) := u_0, &\\[1.2mm]
		v_h(\cdot,t;\zeta, \mcl P) := S_h(t,t_n;\zeta)v_h(\cdot,t_n; \zeta, \mcl P) & \text{for } n = 0, 1, \ldots, N-1 \text{ and } t \in (t_n, t_{n+1}].
	\end{dcases}
\end{equation}

\begin{theorem} \label{T:generalconvergence}
	Assume $u_0 \in BUC(\RR^d)$, \eqref{A:F}, \eqref{A:HC2}, and $S_h$, $W_h$, and $\rho_h$ satisfy \eqref{A:constantcommute} - \eqref{A:consistency}. Let $\{ \mcl P_h \}_{h > 0}$ be a family of partitions of $[0,T]$ such that $\abs{\mcl P_h} \le \rho_h$ for all $h > 0$, and define $u_h := v_h(\cdot; W_h, \mcl P_h)$. Then, as $h \to 0$, $u_h$ converges locally uniformly to the pathwise viscosity solution $u$ of \eqref{E:eq}.
\end{theorem}

The proof of Theorem \ref{T:generalconvergence}, which, as in \cite{BS}, makes use of the method of half-relaxed limits, will be postponed until the end of this section. In the following sub-sections, we demonstrate its utility in a variety of contexts.

\subsection{Finite difference schemes}

Define, for $x \in \RR^d$ and $y \in \ZZ^d \backslash \{0\}$, the discrete derivatives
\begin{equation} \label{E:discretederiv}
	\begin{split}
	&D^+_{h,y} u(x) := \frac{u(x + h y) - u(x)}{h|y|}, \quad D^-_{h,y} u(x) := \frac{u(x) - u(x - h y)}{h|y|}, \\
	&\text{and} \quad D^2_{h,y}u(x) := D^+_{h,y} D^-_{h,y}u(x) = \frac{u(x+hy) + u(x - hy) - 2u(x)}{h^2|y|^2}.
	\end{split}
\end{equation}
Observe that there exists a universal constant $C > 0$ such that, if $u \in C^{1,1}(\RR^d)$, $h > 0$, and $y \in\ZZ^d \backslash \{0\}$, then
\begin{equation}
	\nor{ D^{\pm}_{h,y} u - Du \cdot \frac{y}{|y|}}{\oo} \le C \nor{D^2 u}{\oo} h, \label{E:discreteexactfirst}
\end{equation}
and, if $u \in C^2(\RR^d)$,
\begin{equation} \label{E:discreteexactsecond}
	\nor{D^2_{h,y} u - D^2 u \frac{y}{|y|} \cdot \frac{y}{|y|} }{\oo}
	\le C \sup_{|x_1 - x_2| \le h} \abs{ D^2u(x_1) - D^2 u(x_2)}.
\end{equation}
For some fixed $N \in \NN$, define
\begin{align*}
	\ZZ^d_N &:= \left\{ y \in \ZZ^d : \max_{i=1,2,\ldots, d} \abs{ y_i} \le N \right\},
	\quad
	D^{\pm}_{h,N} := \{ D^{\pm}_{h,y} \}_{y \in \ZZ^d_N \backslash\{0\}},
	\quad
	D_{h,N} := \pars{ D^+_{h,N} \; D^-_{h,N} },\\
	&\text{and} \quad
	D^2_{h,N} := \{ D^2_{h,y} \}_{y \in \ZZ^d_N \backslash\{0\}}.
\end{align*}
Then, for some given functions 
\[
	H_h \in C^{0,1}(\RR^{(2N+1)^d - 1} \times \RR^{(2N+1)^d-1} \times \RR)
\quad \text{and} \quad
	F_h \in C^{0,1}(\RR^{2N+1)^d - 1} \times \RR^{(2N+1)^d - 1} \times \RR^{(2N+1)^d-1}),
\]
the scheme operators for finite difference approximations take the form
\begin{equation} \label{E:explicitschemeop}
	S_h(t,s; \zeta)u(x) := u(x) + F_h\pars{ D^2_h u(x), D_h u(x)}(t-s) + H_h\pars{ D_h u(x), \zeta(t) - \zeta(s)}.
\end{equation}
Properties \eqref{A:constantcommute} and \eqref{A:translatecommute} are immediate, while the question of whether \eqref{E:explicitschemeop} satisfies \eqref{A:monotonicity} or \eqref{A:consistency} is reduced to routine calculations involving $F_h$ and $H_h$.

\subsubsection{Hamilton-Jacobi equations} \label{SS:HamiltonJacobi}

We first study the first-order setting, for which $F = F_h = 0$, and assume, in addition to \eqref{A:HC2}, that
\begin{equation} \label{E:quantnumericalconsistency}
	\left\{
	\begin{split}
		&D_{p,q} H_h(\cdot,\cdot,\Delta \zeta) \le C\pars{  \abs{ \Delta \zeta} + h} \quad \text{for some $C = C_L > 0$ and all $h > 0$ and $\Delta \zeta \in \RR^m$,}\\[1.2mm]
		&\text{and }H_h \pars{ p, p, \Delta \zeta} = \sum_{i=1}^m H^i(p)(\Delta \zeta)^i \quad \text{for all $h > 0$, $p \in \RR^{(2N+1)^d - 1}$, and $\Delta \zeta \in \RR^m$.}
	\end{split}
	\right.
\end{equation}
In order for monotonicity to hold, the Lipschitz bounds in \eqref{E:quantnumericalconsistency} are made more precise. Let elements of $\RR^{(2N+1)^d - 1}$ be labeled by $\left\{ p_y \right\}_{y \in \ZZ^d_N \backslash\{0\}}$, and assume that, for some $C = C_L > 0$, $\theta \in [0,1]$, and $\lambda_0 > 0$,
\begin{equation}\label{E:quantnumericalmonotonicity}
	\left\{
	\begin{split}
		&\sum_{y \in \ZZ^d_N \backslash \{0\} } \frac{1}{|y|} \pars{ \frac{\del H_h}{\del q_y} - \frac{\del H_h}{\del p_y} } \le \frac{1 - \theta}{\lambda_0} \abs{ \Delta \zeta} + \theta h \quad \text{and} \\[1.2mm]
		&\frac{\del H_h}{\del q_y} - \frac{\del H_h}{\del p_{-y}} \ge C\pars{ h - \frac{\abs{ \Delta \zeta}}{\lambda_0}} \quad \text{for all } y \in \ZZ^d_N \backslash\{0\}.
	\end{split}
	\right.
\end{equation}

\begin{lemma} \label{L:consistencycheck}
	Suppose that $H$ satisfies \eqref{A:HC2} and $H_h$ satisfies \eqref{E:quantnumericalconsistency}. Then there exists $C = C_L > 0$ such that, whenever $\zeta \in C([0,T], \RR^m)$, $\mathrm{osc}(\zeta,s,t) \le \lambda_0 h$ for some $s,t \in I$, and $\Phi \in C(I, C^{1,1}(\RR^d))$ is a solution of
	\[
		d\Phi = \sum_{i=1}^m H^i(D\Phi) \circ d\zeta^i \quad \text{in } \RR^d \times I,
	\]
	then
	\[
		\nor{S_h(t,s;\zeta)\Phi(\cdot,s) - \Phi(\cdot, t)}{\oo} \le C \nor{D^2\Phi}{\oo} h^2.
	\]
	
	If, in addition, $H_h$ satisfies \eqref{E:quantnumericalmonotonicity}, then, whenever $u_1, u_2 \in (B)UC(\RR^d)$ with $u_1 \le u_2$ and $\mathrm{osc}(\zeta,s,t) \le \lambda_0 h$,
	\[
		S_h(t,s;\zeta)u_1 \le S_h(t,s;\zeta)u_2.
	\]
\end{lemma}

Motivated by the above result, the schemes for first-order equations in Sections \ref{S:pathwise} and \ref{S:regularrates}, for which we obtain explicit error estimates, will be assumed to satisfy the conclusions of Lemma \ref{L:consistencycheck}. In fact, the smoothness assumption \eqref{A:HC2} is not needed in the proof of Lemma \ref{L:consistencycheck}, and the quantitative convergence results in those sections can be proved under the more general hypotheses \eqref{A:Hdiffconvex} and \eqref{A:Hlingrowth}.

\begin{proof}[Proof of Lemma \ref{L:consistencycheck}]
Let $\Phi \in C(I, C^{1,1}(\RR^d))$ be as in the statement of the lemma. Then there exists $C > 0$ depending only on $\max_{|p| \le L} \abs{ DH(p)}$ such that, for all $s, t \in I$,
\[
	\nor{\Phi(\cdot, t) - \Phi(\cdot,s) - \sum_{i=1}^m H^i(D\Phi(\cdot,s)) \pars{ \zeta^i(t) - \zeta^i(s)} }{\oo} \le C \nor{D^2\Phi}{\oo} \abs{\zeta(t) - \zeta(s)}^2.
\]
Therefore,
\[
	\abs{ S_h(t,s; \zeta)\Phi(\cdot,s)(x) - \Phi(x,t)} \le C \nor{D^2\Phi}{\oo} \pars{ h^2 + \abs{\zeta(t) - \zeta(s)}h + \abs{\zeta(t) - \zeta(s)}^2} \le C(1 + \lambda_0 + \lambda_0^2)h^2.
\]
Meanwhile, if $\mcl S_h: \RR^{(2N+1)^d} \to \RR$ is the map implicitly defined by
\[
	\mcl S_h \pars{ \left\{u(x+y) \right\}_{y \in \ZZ^d_N} } = S_h(t,s; \zeta)u(x),
\]
then \eqref{E:quantnumericalmonotonicity} implies that $\mcl S_h$ is increasing in each of its arguments as long as $\mathrm{osc}(\zeta,s,t) \le \lambda_0 h$.
\end{proof}

We now mention two specific examples. The first is the analogue of the Lax-Friedrichs scheme for scalar conservation laws discussed in the introduction. Here, $H_h$ is defined, for some $\theta \in (0,1]$, by
\[
	H_h(p,q,\Delta \zeta) := H\pars{ \frac{p+q}{2}}\Delta \zeta + \frac{\theta h}{2d} \sum_{k=1}^d \pars{ q_k - p_k},
\]
where the vector $(p,q) \in \RR^d \times \RR^d$ stands for the discrete derivatives
\begin{align*}
	&p = \pars{ D^+_{h, e_1}, D^+_{h, e_2}, \ldots, D^+_{h, e_d} },
	\quad
	q = \pars{ D^-_{h, e_1}, D^-_{h, e_2}, \ldots, D^-_{h, e_d} }, \\
	&e_k := (0, 0, \ldots, 0, \underbrace{1}_{k}, 0, \ldots, 0) \quad \text{for } k = 1, 2, \ldots, d.
\end{align*}
A calculation verifies that \eqref{E:quantnumericalconsistency} and \eqref{E:quantnumericalmonotonicity} are satisfied with $\lambda_0 := \frac{\theta}{d \nor{DH}{\oo}}$.

If $d = 1$, the different regions of monotonicity of $H$ may be exploited to create upwind schemes. As a simple example, assume that $H \ge H(0) = 0$ and $H$ is increasing for $p > 0$ and decreasing for $p < 0$, and define
\[
	H_h(p,q,\Delta \zeta) :=
	\left[ H( p_+) + H(- q_-) \right] (\Delta \zeta)_+ - \left[ H( q_+) + H(-p_-) \right] (\Delta \zeta)_-.
\]
Then \eqref{E:quantnumericalconsistency} and \eqref{E:quantnumericalmonotonicity} hold with $\theta = 0$ and $\lambda_0 := \frac{1}{2\nor{H'}{\oo}}$.

As far as the approximating paths $W_h$ are concerned, Lemma \ref{L:consistencycheck} implies that \eqref{A:monotonicity} and \eqref{A:consistency} will hold, with $k = 2$, if $\rho_h$ and $W_h$ satisfy
\begin{equation}
	\sup_{0 \le t-s \le \rho_h} \abs{ W_h(t) - W_h(s)} \le \lambda_0 h \quad \text{and} \quad \lim_{h \to 0} \frac{h^2}{\rho_h} = 0. \label{A:extrarho}
\end{equation}
If $W_h$ is smooth, then
\[
	\sup_{0 \le t-s \le \rho_h} \abs{ W_h(t) - W_h(s)} \le \nor{\dot W_h}{\oo} \rho_h.
\]
Let $\omega: [0,\oo) \to [0,\oo)$ be the modulus of continuity for $W$. For many standard approximations of $W$, there exists some increasing function $h \mapsto \eta_h$ satisfying $\lim_{h \to 0^+} \eta_h = 0$ and some $C > 0$ such that
\begin{equation}\label{E:slowderivative}
	\nor{\dot W_h}{\oo} \le C \frac{\omega(\eta_h)}{\eta_h}.
\end{equation}
For example, $W_h$ may be the piecewise linear interpolation of $W$ with step-size $\eta_h$, or the convolution of $W$ with a standard mollifier supported in an interval of radius $\eta_h$. Then the first part of \eqref{A:extrarho} may be replaced with the slightly stronger assumption
\begin{equation} \label{A:extraextrarho}
	\frac{C\omega(\eta_h) \rho_h}{h \eta_h} \le \lambda_0.
\end{equation}
To be more explicit, suppose that $W \in C^\alpha([0,T], \RR^m)$ and, for some $\gamma > 0$, $\eta_h = (\rho_h)^\gamma$. Then \eqref{A:extraextrarho} will hold if $\rho_h$ is defined by
\[
	\lambda := \frac{C [W]_{\alpha,T}(\rho_h)^{1 - \gamma + \alpha \gamma}}{h} \le \lambda_0.
\]
This yields
\[
	\frac{h^2}{\rho_h} \approx (\rho_h)^{1 - 2 \gamma + 2\alpha \gamma},
\]
so that \eqref{A:extrarho} will be satisfied if
\[
	0 < \gamma < \frac{1}{2(1-\alpha)}.
\]
If $\alpha > \frac{1}{2}$, then $\gamma$ is allowed to be $1$, and in particular, it is natural to define $W_h$ to be the piecewise linear interpolation of $W$ on a partition of step-size $\eta_h = \rho_h$. Notice also that paths in $C^\alpha$ for such $\alpha$ have quadratic variation equal to $0$.

However, for $\alpha \le \frac{1}{2}$, $\gamma$ is forced to be less than $1$, and so we must make $W_h$ a milder approximation. The work in the subsequent sections suggests that choosing $\gamma = \frac{1}{2}$ gives the best rate of convergence regardless of the regularity of the path $W$.

\subsubsection{A second order example}

Verifying \eqref{A:monotonicity} and \eqref{A:consistency} is more complicated for finite difference approximations of second order equations. Rather than stating very general assumptions on $F_h$ or $H_h$, we perform these calculations for a specific scheme. More examples can be formed by adapting the results of \cite{KT1, KT2}.

Assume for simplicity that $d = 1$, $H \in C^3(\RR, \RR^m)$, and that $F$ depends only on $u_{xx}$, and define, for some $\epsilon _h > 0$,
\[
	H_h(p,q,\Delta \zeta) := H\pars{ \frac{p+q}{2}}\Delta \zeta \quad \text{and} \quad F_h(X) = F(X) + \epsilon_h X \quad \text{for} \quad X = D^2_{h,1} u \text{ and } (p,q) = D_{h,1} u.
\]
Note that the ellipticity condition \eqref{A:F} means that $F$ is increasing, and so a routine calculation shows that $S_h$, $W_h$, and $\rho_h$ satisfy \eqref{A:monotonicity} if
\[
	\rho_h := \lambda h^2 \quad \text{with} \quad \lambda \le \frac{1}{2 \nor{F'}{\oo}}
\]
and
\begin{equation}\label{E:LFLipschitzpath}
	\nor{\dot W_h}{\oo} \le \frac{2}{\nor{H'}{\oo}} \cdot \frac{\epsilon_h}{h}.
\end{equation}
Now let $\Phi_h \in C(I, C^3(\RR))$ and $\phi \in C^3(\RR)$ be as in \eqref{A:consistency}. Observe that it is possible to find such a solution because of the added regularity for $H$, and that
\[
	\sup_{h > 0} \pars{ \nor{\Phi_{h,xx}}{\oo} + \nor{\Phi_{h,xxx}}{\oo} } < \oo.
\]
Then, for some $C > 0$ depending only on $\nor{H'}{\oo}$, and for all $\rho \in (0, \lambda h^2)$,
\begin{align*}
	&\abs{ \Phi_h(x,t + \rho) - \Phi_h(x,t) - \sum_{i=1}^m H^i(\Phi_{h,x}(x,t))(W_h(t+\rho) - W_h(t)) } \\
	&\le C \pars{ \nor{\Phi_{h,xx}}{\oo} \max_{t \le s \le t + \rho} \abs{ W_h(s) - W_h(t)}^2 } \le C \lambda \nor{\Phi_{h,xx}}{\oo} (\epsilon_h)^2 \rho.
\end{align*}
The estimates \eqref{E:discreteexactfirst} and \eqref{E:discreteexactsecond} then imply that, for $s_h$ and $t_h$ as in \eqref{A:consistency},
\begin{align*}
	&\abs{S_h(t_h,s_h; W_h)\Phi_h(\cdot,s_h)(x) - \Phi_h(x,t_h) - (t_h - s_h) F(\phi_{xx}(x,t)) }\\
	&\le C\rho_h \cdot \pars{ \nor{\Phi_{h,xx}}{\oo} \epsilon_h+ \nor{\Phi_{h,xxx}}{\oo} h + \nor{\Phi_{h,xx}(\cdot,s_h) - \phi_{xx}}{\oo}},
\end{align*}
and so \eqref{A:consistency} holds if $\lim_{h \to 0} \epsilon_h = 0$. This, in turn, requires that
\[
	\lim_{h \to 0} h \nor{\dot W_h}{\oo} = 0,
\]
or that $W_h$ satisfies \eqref{E:slowderivative} with $\eta_h$ such that
\[
	\lim_{h \to 0} \frac{h\omega(\eta_h)}{\eta_h} = 0.
\]
Taking $W \in C^\alpha([0,T], \RR^m)$ and $\eta_h = (\rho_h)^\gamma = \lambda^\gamma h^{2\gamma}$ for some $\gamma > 0$ as a concrete example, this leads once more to the restriction
\[
	0 < \gamma < \frac{1}{2(1-\alpha)}.
\]

\subsection{Other approximations}

\subsubsection{Stability for \eqref{E:eq}}

The proof of Theorem \ref{T:generalconvergence} is a generalization of the argument that \eqref{E:eq} is stable with respect to perturbations in the data. In fact, Theorem \ref{T:generalconvergence} recovers these stability properties.

Suppose that
\begin{equation} \label{A:epsilonfamilies}
	\left\{
	\begin{split}
		&u_0^\epsilon \in C^{0,1}(\RR^d), \quad W^\epsilon, W \in C([0,T]; \RR^m), \quad H^\epsilon, H \in C^2(\RR^d; \RR^m), \quad F^\epsilon, F \text{ satisfy \eqref{A:F}}, \\[1.2mm]
		&\text{and } \lim_{\epsilon \to 0} \pars{ \nor{u_0^\epsilon - u_0}{\oo},  \nor{W^\epsilon - W}{\oo},  \nor{H^\epsilon - H}{C^2}, \nor{F^\epsilon - F}{\oo} } = 0,
	\end{split}
	\right.
\end{equation}
and let $u^\epsilon \in BUC(\RR^d \times [0,T])$ be the unique solution of
\begin{equation}\label{E:perturbedeq}
	du^\epsilon = F^\epsilon(D^2u^\epsilon, Du^\epsilon)\;dt + \sum_{i=1}^m H^{i, \epsilon}(Du^\epsilon) \circ dW^{i, \epsilon} \quad \text{in } \RR^d \times (0,T] \quad \text{and} \quad u^\epsilon(\cdot, 0) = u^\epsilon_0 \quad \text{on } \RR^d.
\end{equation}

\begin{theorem}\label{T:stability}
	Assume \eqref{A:epsilonfamilies} and let $u^\epsilon$ and $u$ solve respectively \eqref{E:perturbedeq} and \eqref{E:eq}. Then, as $\epsilon \to 0$, $u^\epsilon$ converges locally uniformly to $u$.
\end{theorem}

\begin{proof}
	The comparison principle implies that the solution operator for \eqref{E:perturbedeq} is contractive, and therefore, it suffices to assume that $u_0^\epsilon = u_0$ for all $\epsilon > 0$.
	
	For $s \le t$, $\zeta \in C([0,T]; \RR^m)$, and $h > 0$, let $S_\epsilon(t,s; \zeta): BUC(\RR^d) \to BUC(\RR^d)$ be the solution operator for \eqref{E:perturbedeq} driven by the path $\zeta$ instead of $W^\epsilon$. Properties \eqref{A:constantcommute} and \eqref{A:translatecommute} are readily verified, and, letting $\rho_h = \rho_\epsilon$ be arbitrary and setting $W_h = W^\epsilon$, \eqref{A:monotonicity} follows immediately from the comparison principle. 
	
	Finally, in view of the uniform bound for $D^2 H^\epsilon$, for any interval $I \subset [0,T]$ and solution $\Phi \in C(I, C^2(\RR^d))$ of \eqref{E:HJpart}, there exists a family of solutions $\Phi^\epsilon \in C(I, C^2(\RR^d))$ solving \eqref{E:HJpart} with the Hamiltonian $H^\epsilon$ and path $W^\epsilon$, converging in $C(I, C^2(\RR^d))$ to $\Phi$ as $\epsilon \to 0$. This can be seen using the method of characteristics, as in Section \ref{S:definition}. Therefore, \eqref{A:consistency} is a consequence of Definition \ref{D:smoothH} and the local uniform convergence of $F^\epsilon$ to $F$. Theorem \ref{T:generalconvergence} now gives the result.	
\end{proof}

\subsubsection{A mixing formula}
It is also possible to derive general Trotter-Kato type mixing formulas for \eqref{E:eq}. Here, we present a specific example. A different approach to the following can be found in the work of Gassiat and Gess \cite{GG}.

Assume, in addition to \eqref{A:F}, that 
\[
	F \in C^{1,1}(S^d \times \RR^d) \quad \text{and} \quad H \in C^4(\RR^d, \RR^m),
\]
and, for $\zeta \in C([0,T], \RR^m)$, let $S_F(t): BUC(\RR^d) \to BUC(\RR^d)$ and $S_H(t,s; \zeta): BUC(\RR^d) \to BUC(\RR^d)$ be the solution operators for respectively
\[
	u_t = F(D^2 u, Du) \quad \text{and} \quad du = \sum_{i=1}^m H^i(Du) \circ d\zeta^i.
\] 
Define
\[
	S_h(t,s;\zeta) = S_F(t-s) S_H(t,s; \zeta).
\]

\begin{theorem}
	For any sequence of approximating paths $\{W_h\}_{h > 0}$ and modulus $h \to \rho_h$ satisfying \eqref{A:approxpath}, the triple $(S_h, W_h, \rho_h)$ satisfies \eqref{A:constantcommute} - \eqref{A:consistency}.
\end{theorem}

\begin{proof}
	Properties \eqref{A:constantcommute} - \eqref{A:monotonicity} are immediate from the definitions of the above objects. Let $I \subset [0,T]$, $s_h, t_h \in I$, $\Phi_h \in C(I; C^4(\RR^d))$, and $\phi \in C^4(\RR^d)$ be as in \eqref{A:consistency}. Such a solution $\Phi$ exists in view of the additional regularity assumed for $H$.
	
	For any $x \in \RR^d$,
	\[
		S_h(t_h, s_h; W_h) \Phi_h(\cdot, s_h)(x) - \Phi_h(x,t_h)
		= S_F(t_h - s_h)\Phi_h(\cdot,t_h)(x) - \Phi_h(x,t_h).
	\]
	Define $\phi_h := \Phi_h(\cdot,t_h)$, which satisfies
	\[
		R := \sup_{h > 0} \nor{\phi_h}{C^4(\RR^d)} < \oo \quad \text{and} \quad \lim_{h\to 0} \nor{\phi_h - \phi}{C^2(\RR^d)} = 0,
	\]
	and let
	\[
		z_h(x,t) := \phi_h(x) + t F(D^2\phi_h(x), D\phi_h(x)).
	\]
	Then, for some universal constant $C > 0$, $z_h$ is a viscosity super-solution of
	\[
		z_{h,t} \ge F(D^2 z_h, Dz_h) - C \nor{F}{C^{1,1}(\RR^d)} R \rho_h \quad \text{in } \RR^d \times [0,\rho_h],
	\]
	so that, for all $\rho \in (0, \rho_h)$,
	\[
		\sup_{x \in \RR^d} \pars{S_F(\rho)\phi_h(x) - z_h(x,\rho) } \le C \nor{F}{C^{1,1}(\RR^d)} R\rho_h \rho.
	\]
	A similar argument, using that $z_h$ satisfies an analogous viscosity sub-solution property, gives a lower bound, whence
	\[
		\abs{S_F(t_h - s_h)\phi_h(x) - \phi_h(x) - (t_h - s_h) F(D^2 \phi_h(x), D\phi_h(x)) } \le C \nor{F}{C^{1,1}(\RR^d)} R\rho_h (t_h - s_h).
	\]
	Property \eqref{A:consistency} now follows, with $k = 4$, from the fact that
	\[
		\lim_{h \to 0} F(D^2 \phi_h, D\phi_h) = F(D^2 \phi, D\phi) \quad \text{uniformly.}
	\]
	\end{proof}

\subsection{The proof of Theorem \ref{T:generalconvergence} }

Define
\[
	u^*(x,t) = \limsup_{h \to 0, (y,s) \to (x,t)} u_h(y,s)
	\quad \text{and} \quad
	u_*(x,t) = \liminf_{h \to 0, (y,s) \to (x,t)} u_h(y,s).
\]
The functions $u^*$ and $u_*$, called the half-relaxed limits of $u_h$, are respectively upper- and lower- semicontinuous. Furthermore, $u_* \le u^*$ on $\RR^d \times [0,T]$ and $u_*(\cdot,0) \le u_0 \le u^*(\cdot,0)$ on $\RR^d$. The goal will be to show that $u_* = u^*$, which yields the local uniform convergence of $u_h$ and the fact that the limit $u$ solves \eqref{E:eq}.

{\it Step 1: Finiteness of $u^*$ and $u_*$.} Observe that, for any constant $k \in \RR$, the function
\[
	\Phi_h(x,t) = k + \sum_{i=1}^m H^i(0)W_h^i(t)
\]
is a smooth solution of \eqref{E:HJpart} for all $(x,t) \in \RR^d \times [0,T]$. Therefore, in view of \eqref{A:monotonicity} and \eqref{A:consistency}, 
\[
	u_h(x,t) \le \nor{u_0}{\oo} + \sum_{i=1}^m H^i(0) W^i_h(t) + T(F(0,0) + 1)
\]
for all sufficiently small $h > 0$, and so $u^*(x,t) < \oo$ for all $(x,t) \in \RR^d \times [0,T]$. A similar argument gives $u_* > -\oo$.

{\it Step 2: The solution inequalities.} In this step, we demonstrate that $u^*$ and $u_*$ satisfy respectively the sub- and super-solution properties in Definition \ref{D:smoothH} for equation \eqref{E:eq}. Only the argument for $u^*$ is presented, since the proof for $u_*$ is similar.
	
Assume that $(x_0,t_0) \in \RR^d \times (0,T]$, $I \ni t_0$, $\psi \in C^1([0,T])$, $\Phi \in C(I, C^k(\RR^d))$ solves \eqref{E:HJpart} with
\[
	\max_{j = 2, 3, \ldots, k} \sup_{t \in I} \nor{D^j \Phi(\cdot,t)}{\oo} < \oo,
\]
and $u^*(x,t) - \Phi(x,t) - \psi(t)$ attains a local maximum at $(x_0,t_0)$. As discussed in Section \ref{S:definition}, it may be assumed that this maximum is strict in $\RR^d \times I$, and that
\begin{equation}
	\lim_{|x| \to +\oo} \frac{\Phi(x,t)}{|x|} = +\oo \qquad \text{uniformly for $t \in I$}. \label{E:quadgrowth}
\end{equation}
	
The definition of $u^*$ implies that there exist $y_h \in \RR^d$ and $s_h \in [0,T]$ such that
\[
	\lim_{h \to 0} \pars{ y_h, s_h, u_h(y_h,s_h) } = (x_0, t_0, u^*(x_0,t_0)).
\]

The method of characteristics and the fact that $\lim_{h \to 0} \nor{W_h - W}{\oo} = 0$ yield the existence of a subinterval of $I$ containing $t_0$, relabeled as $I$ for convenience, such that, for all $h > 0$, there exists a solution $\Phi_h \in C(I, C^k(\RR^d))$ of
\[
	d\Phi_h = \sum_{i=1}^m H^i(D\Phi_h) \circ dW_h \quad \text{in } \RR^d \times I \quad \text{and} \quad \Phi_h(\cdot,t_0) = \Phi(\cdot,t_0) \quad \text{in } \RR^d
\]
that satisfies \eqref{E:quadgrowth} uniformly in $h$, and $\Phi_h$ converegs to $\Phi$ in $C(I, C^k(\RR^d))$ as $h \to 0$. It follows that
\[
	u_h(x,t) - \Phi_h(x,t) - \psi(t)
\]
attains a global maximum at $(\hat y_h, \hat s_h)$ over $\RR^d \times \oline{I}$ such that $\{\hat y_h\}_{h > 0}$ is bounded. This gives, in particular,
\[
	u_h(y_h,s_h) - \Phi_h(y_h,s_h) - \psi(s_h)
	\le u_h(\hat y_h, \hat s_h) - \Phi_h(\hat y_h, \hat s_h) - \psi(\hat s_h).
\]
Let $(\hat x, \hat t)$ be a limit point of the sequence $\{ (\hat y_h, \hat s_h) \}_{h > 0}$. Taking $h \to 0$ along the appropriate subsequence above results in the inequality
\[
	u^*(x_0, t_0) - \Phi(x_0, t_0) - \psi(t_0) \le u^*(\hat x, \hat t) - \Phi(\hat x, \hat t) - \psi(\hat t).
\]
The strictness of the original maximum then implies that $\lim_{h \to 0} (\hat y_h, \hat s_h) = (x_0,t_0)$. 

Because $\abs{ \mcl P_h} \le \rho_h \xrightarrow{h \to 0} 0$, it follows that, for sufficiently small $h$, there exists $t_n \in \mcl P_h$ such that $t_n < \hat s_h \le t_{n+1}$ and $t_n \in I$. Then, for all $x \in \RR^d$,
\begin{equation}\label{E:preschemeineq}
	u_h(x,t_n) \le u_h(\hat y_h, \hat s_h) + \Phi_h(x,t_n) - \Phi_h( \hat y_h, \hat s_h) + \psi(t_n) - \psi(\hat s_h). 
\end{equation}
Applying the operator $S_h(\hat s_h, t_n; W_h)$ to both sides of \eqref{E:preschemeineq}, using \eqref{A:monotonicity} and the fact that $0 < \hat s_h - t_n \le \rho_h$, and rearranging terms yields
\[
	\frac{\psi(\hat s_h) - \psi(t_n)}{\hat s_h - t_n} \le \frac{ S_h(\hat s_h, t_n; W_h)\Phi_h(\cdot,t_n)(\hat y_h) - \Phi_h(\hat y_h, \hat s_h)}{\hat s_h - t_n}.
\]
Sending $h \to 0$ and using \eqref{A:consistency} gives $\psi'(t_0) \le F(D^2 \Phi(x_0,t_0), D\Phi(x_0,t_0))$, as desired.
	
{\it Step 3: Initial data.} We now prove that $u^*(x,0) = u_0(x) = u_*(x,0)$. Only the first equality is considered, and since $u^*(x,0) \ge u_0(x)$, it suffices to show that $u^*(x,0) \le u_0(x)$.
	
Let $\phi \in C^k(\RR^d)$ be such that
\[
	R := \max_{j = 2, 3, \ldots, k} \nor{D^j \phi}{\oo} < \oo
\]
and $u_0 \le \phi$ on $\RR^d$, and let $I \ni 0$ and $\Phi \in C(I, C^k(\RR^d))$ be a solution of \eqref{E:HJpart} with $\Phi(\cdot,0) = \phi$. Define $\phi_h \in C^{0,1}(\RR^d \times [0,T])$ as in \eqref{E:approxsolution} with initial condition $\phi_h(\cdot,0) = \phi$, path $W_h$, and partition $\mcl P_h$. Then \eqref{A:monotonicity} and \eqref{A:consistency} yield, for some $C > 0$ depending only on $R$ and $\nor{DF}{\oo}$, and for any $(y,s) \in \RR^d \times I$ and sufficiently small $h$,
\[
	u_h(y,s) \le \phi_h(y,s) \le \Phi(y,s) + Cs.
\]
Sending $(y,s) \to (x,0)$ and $h \to 0$, this becomes $u^*(x,0) \le \phi(x)$, completing the argument since $\phi$ was arbitrary.
	
{\it Step 4: The comparison principle.} In view of the comparison principle \eqref{E:comparison}, $u^*(x,t) \le u_*(x,t)$ for all $(x,t) \in \RR^d \times [0,T]$. Therefore $u^* = u_*$, and the result is proved.

\section{The pathwise estimate} \label{S:pathwise}

The remaining sections focus on deriving quantitative error estimates for schemes in the first-order setting. We will henceforth always assume that $H$ satisfies \eqref{A:Hdiffconvex} and $u_0$ satisfies \eqref{A:initialcondition} (and thus, without loss of generality, $H$ satisfies \eqref{A:Hlingrowth}). Also, in addition to \eqref{A:constantcommute} and \eqref{A:translatecommute}, the schemes in this part of the paper will be required to satisfy the following quantitative versions of \eqref{A:monotonicity} and \eqref{A:consistency}: for some $\lambda_0 > 0$,
\begin{equation}\label{A:quantmonotonicity}
	\text{if } u_1 \le u_2 \text{ and } \mathrm{osc}(\zeta,s,t) \le \lambda_0 h, \quad \text{then } S_h(t,s; \zeta)u_1 \le S_h(t,s; \zeta)u_2,
\end{equation}
and
\begin{equation}\label{A:quantconsistency}
	\left\{
	\begin{split}
	&\text{there exists $C = C_L > 0$ such that, if $\zeta \in C([0,T], \RR^m)$, $\Phi \in C(I, C^{1,1}(\RR^d))$}\\[1.2mm]
	&\text{is a solution of $d\Phi = \sum_{i=1}^m H^i(D\Phi) \circ d \zeta^i$ in $\RR^d \times I$, and $\mathrm{osc}(\zeta,s,t) \le \lambda_0 h$, then} \\[1.2mm]
	&\nor{S_h(t,s; \zeta)\Phi(\cdot,s) - \Phi(\cdot,t)}{\oo} \le C \nor{D^2\Phi}{\oo}h^2.
	\end{split}
	\right.
\end{equation}
This is motivated by the properties obtained in Lemma \ref{L:consistencycheck} for the finite difference approximations discussed in subsection \ref{SS:HamiltonJacobi}.

Fix a partition
\[
	\mcl P = \left\{0 = t_0 < t_1 < t_2 < \cdots < t_N = T\right\}
\]
of $[0,T]$, set $(\Delta t)_n := t_{n+1} - t_n$, and let $\zeta: [0,T] \to \RR^m$ be any continuous path satisfying
\begin{equation}\label{A:pathrestrict}
	\left\{
	\begin{split}
		&\zeta(0) = 0, \quad \text{$\zeta$ is affine on $[t_n,t_{n+1}]$ for every $n = 0, 1, 2, \ldots, N-1$, and}  \\[1.2mm]
		&\max_{n=0,1,2,\ldots,N-1} \abs{ \zeta(t_{n+1}) - \zeta(t_n)} \le \lambda_0 h.
	\end{split}
	\right.
\end{equation}
In this section, we obtain an estimate for the error between the viscosity solution $v$ of
\begin{equation}\label{E:zetaeq}
	v_t = \sum_{i=1}^m H^i(Dv)\dot \zeta^i(t) \quad \text{in } \RR^d \times (0,T] \quad \text{and} \quad v(\cdot,0) = u_0 \quad \text{on } \RR^d
\end{equation}
and the approximate solution $v_h(\cdot; \zeta, \mcl P)$ given by \eqref{E:approxsolution}, which, for convenience, we define again here:
\begin{equation} \label{E:quantapproxsolution}
	\begin{dcases}
		v_h(\cdot,0; \zeta, \mcl P) := u_0, &\\[1.2mm]
		v_h(\cdot,t;\zeta, \mcl P) := S_h(t,t_n;\zeta)v_h(\cdot,t_n; \zeta, \mcl P) & \text{for } n = 0, 1, \ldots, N-1 \text{ and } t \in (t_n, t_{n+1}].
	\end{dcases}
\end{equation}

\begin{theorem} \label{T:pathwise}
Assume \eqref{A:Hdiffconvex}, \eqref{A:initialcondition}, and \eqref{A:Hlingrowth}. Then there exists $C = C_{L}> 0$ such that, if $S_h$ satisfies \eqref{A:constantcommute}, \eqref{A:translatecommute}, \eqref{A:quantmonotonicity}, and \eqref{A:quantconsistency}, $\zeta$ and $\mcl P$ satisfy \eqref{A:pathrestrict}, and $v$ and $v_h$ are as in \eqref{E:zetaeq} and \eqref{E:quantapproxsolution} with $\nor{Du_0}{\oo} \le L$, then, for all $\epsilon, h > 0$,
	\[
		\sup_{ (x,t) \in \RR^d \times [0,T]} \abs{v_h(x,t; \zeta, \mcl P) - v(x,t)} \le \frac{1}{\epsilon} \sum_{n=0}^{N-1} (\Delta t_n)^2 + C \sqrt{N} h + \max_{s,t \in [0,T]} \left\{ C \abs{\zeta (s) - \zeta(t)} - \frac{|s-t|^2}{2\epsilon} \right\} .
	\]
\end{theorem}

Before proving Theorem \ref{T:pathwise}, we state some regularity estimates for $v$ and $v_h$. First, the monotonicity of the scheme operator $S_h$, the comparison principle for \eqref{E:firstordereq}, and the translation invariance of the solution operators for each immediately yield the Lipschitz bounds
\begin{equation}\label{E:vLipschitz}
	\nor{Dv}{\oo}, \; \nor{Dv_h}{\oo} \le L.
\end{equation}
The regularity of $v_h$ and $v$ in the time variable is established by the next result.

\begin{lemma} \label{L:solutionbounds}
	Assume \eqref{A:Hdiffconvex}, \eqref{A:initialcondition}, and \eqref{A:Hlingrowth}. There exists $C = C_{L} > 0$ such that, for all $(x,s,t) \in \RR^d \times [0,T] \times [0,T]$ with $s < t$,
	\begin{equation}\label{E:vtimereg}
		\abs{ v(x,t) - v(x,s) } \le C \mathrm{osc}(\zeta,s, t)
	\end{equation}
	and, for all $m,n \in \{0, 1, 2, \ldots N\}$ with $m < n$,
	\begin{equation}\label{E:vhtimereg}
		\abs{ v_h(x,t_n; \zeta, \mcl P) - v_h(x,t_m; \zeta, \mcl P)} \le  C  \pars{ h \sqrt{n-m}  + \mathrm{osc}(\zeta, t_m, t_n) }. 
	\end{equation}
\end{lemma}

\begin{proof}
	The bound \eqref{E:vtimereg} follows from the cancellation estimates presented in Proposition 7.2 of \cite{Snotes}. To prove \eqref{E:vhtimereg}, observe first that, in view of Lemma \ref{L:distancefunction}(c), there exists $C = C_L > 0$ such that, for all $z \in \RR^d$ and $\delta > 0$,
	\[
		L|z| \le \Phi_\delta(z,t_m, t_m; \zeta) + C\delta.
	\]
	Then \eqref{E:vLipschitz} yields, for all $x,y \in \RR^d$,
	\begin{equation} \label{E:applyscheme}
		v_h(x,t_m; \zeta, \mcl P) \le v_h(y,t_m; \zeta, \mcl P) + L|x-y| \le v_h(y,t_m; \zeta, \mcl P) + \Phi_\delta(x-y,t_m, t_m; \zeta) + C \delta.
	\end{equation}
	Keeping $y$ fixed, we then apply the operator $\prod_{k=m}^{n-1} S_h(t_{k+1}, t_k; \zeta, \mcl P)$ to the left- and right-hand sides of the inequality \eqref{E:applyscheme}, which is preserved because of the monotonicity of this operator implied by \eqref{A:quantmonotonicity} and \eqref{A:pathrestrict}. According to \eqref{E:quantapproxsolution}, the left-hand side becomes $v_h(x,t_n; \zeta, \mcl P)$. Iteratively using \eqref{A:quantconsistency} to compare the right-hand side to $\Phi_\delta(x - y, t_n, t_m; \zeta, \mcl P)$ yields, in view of Lemma \ref{L:distancefunction}(b),
	\begin{align*}
		v_h(x,t_n; \zeta, \mcl P) &\le v_h(y,t_m; \zeta, \mcl P) + \Phi_\delta(x-y, t_n, t_m; \zeta) + C\pars{ \delta + (n-m)\nor{D^2\Phi_\delta}{\oo} h^2}\\
		&\le v_h(y,t_m; \zeta, \mcl P) + \Phi_\delta(x-y, t_n, t_m; \zeta) + C \pars{ \delta + \frac{(n-m) h^2}{\delta}},
	\end{align*}
	as long as $\mathrm{osc}(\zeta,t_m,t_n) \le \delta$. Setting $x=y$ gives
	\[
		v_h(x,t_n; \zeta, \mcl P) - v_h(x,t_m; \zeta, \mcl P) \le C \inf \left\{ \delta + \frac{(n-m) h^2}{\delta} : \delta \ge \mathrm{osc}(\zeta,t_m,t_n) \right\}.
	\]
	If $\mathrm{osc}(\zeta,t_m,t_n) \le h \sqrt{n-m}$, then the right-hand side is optimized by choosing $\delta = h \sqrt{n-m}$. Otherwise, setting $\delta = \mathrm{osc}(\zeta,t_m,t_n)$ gives the result, since in this case,
	\[
		\frac{(n-m) h^2}{\delta} = \frac{(n-m) h^2}{\mathrm{osc}(\zeta,t_m,t_n)} \le h \sqrt{n-m}.
	\]
	The lower bound for $v_h(\cdot,t_n; \zeta, \mcl P) - v_h(\cdot,t_m; \zeta, \mcl P)$ is proved similarly.
\end{proof}

\begin{proof} [Proof of Theorem \ref{T:pathwise}]
	Throughout the proof, to simplify the presentation, we set $v_h(x,t) := v_h(x,t; \zeta, \mcl P)$. Fix a constant $\oline{C} = \oline{C}_L > 0$ to be determined later, and let $\alpha, \mu: [0,T] \to \RR$ be the nondecreasing, lower-semicontinuous, piecewise constant functions defined by $\alpha(0) = \mu(0) = 0$ and
	\[
		\alpha(s) - \alpha(t_n) := [(\Delta t)_n]^2 \quad \text{and} \quad \mu(s) - \mu(t_n) = \oline{C} h^2 \quad \text{for } n = 0, 1, 2,\ldots, N-1 \text{ and }s \in (t_n, t_{n+1}].
	\]
	Choose $\epsilon > 0$ and
	\begin{equation}\label{E:delta}
		\delta > \max \left\{2 \lambda_0 h, \max_{s,t \in [0,T]} \left\{\oline{C} \abs{ \zeta(s) - \zeta(t)} - \frac{|s-t|^2}{2\epsilon} \right\} \right\},
	\end{equation}
	and define the auxiliary function $\Psi: [0,T] \times [0,T] \to \RR$ by
	\begin{equation} \label{E:auxfunc}
		\Psi(s, t) = \sup_{x, y \in \RR^d} \left\{ v_h(x,s) - v(y,t) - \Phi_\delta(x-y, s,t; \zeta) \right\} - \frac{|s - t|^2}{2\epsilon} - \frac{\mu(s)}{\delta}  - \frac{\alpha(s)}{\epsilon},
	\end{equation}
	where $\Phi_\delta$ is the ``distance function'' given in \eqref{E:distancefunction}.
		
	{\it Step 1:} We first prove that, if $\oline{C}$ is sufficiently large, then
	\begin{equation}\label{E:maxPhi}
		\max_{ [0,T]^2 } \Psi = \max \left\{ \max_{s \in [0,T]} \Psi(s,0), \max_{t \in [0,T]} \Psi(0,t) \right\}.
	\end{equation}
	Assume for the sake of contradiction that, for some $\sigma > 0$, $\Psi(s,t) - \sigma t$ attains its maximum in $[0,T] \times [0,T]$ at $(\hat s, \hat t)$ with $\hat s > 0$ and $\hat t > 0$.
	
	The first observation is that, for some $M = M_L > 0$, the supremum in \eqref{E:auxfunc} may be restricted to $x,y \in \RR^d$ satisfying $|x-y| \le M\delta$. This is because, for any $s,t \in [0,T]$ and for some $C' = C'_L > 0$,
	\[
		\sup_{x , y \in \RR^d} \left\{ v_h(x,s) - v(y,t) - \Phi_\delta(x - y , s, t ; \zeta) \right\} 
		\ge \sup_{x \in \RR^d} \left\{ v_h(x,s) -  v(x,t) \right\} - C' \delta,
	\]
	while, if $|x - y| > M\delta$, then \eqref{E:vLipschitz} and Lemma \ref{L:distancefunction}(c) give, for some $C = C_L > 0$,
	\begin{align*}
		v_h(x,s) - v(y,t) - \Phi_\delta(x - y,s, t; \zeta)
		&\le \sup_{x \in \RR^d} \left\{ v_h(x,s) -  v(x,t) \right\}  + L|x-y| - \frac{|x-y|^2}{2(C+1)\delta} + C\delta \\
		&\le \sup_{x \in \RR^d} \left\{  v_h(x,s) -  v(x,t) \right\} - \frac{M^2 \delta}{4(C+1)} + (C + (C+1)L^2) \delta\\
		&< \sup_{x \in \RR^d} \left\{ v_h(x,s) -  v(x,t) \right\} - C' \delta,
	\end{align*}
	where the last inequality holds if $M$ is sufficiently large.
	
	As a result, if $\oline{C}$ is large enough, then $(\hat s, \hat t) \in U_{\delta/2}(W)$. To verify this, we rearrange terms in the inequality $\Psi(\hat s,\hat s) \le \Psi(\hat s, \hat t)$ and use Lemmas \ref{L:distancefunction}(a) and \ref{L:solutionbounds} to obtain, for some $C = C_L > 0$,
	\begin{align*}
		\frac{ | \hat s - \hat t|^2}{2 \epsilon} 
		\le \sup_{|x-y| \le M \delta} \left\{ v_h(y,\hat s) -  v(y,\hat t) + \Phi_\delta(x-y, \hat s , \hat s; \zeta) - \Phi_\delta(x-y, \hat s, \hat t; \zeta) \right\}
		\le C \mathrm{osc}(\zeta, \hat s, \hat t).
	\end{align*}	
	Consequently,
	\begin{align*}
		\oline{C} \mathrm{osc}(\zeta,\hat s, \hat t) &\le \max_{s,t \in [0,T]} \left\{ \oline{C} \mathrm{osc}(\zeta,s,t) - \frac{ |s-t|^2}{2\epsilon} \right\} + \frac{|\hat s - \hat t|^2}{2\epsilon} \\
		&\le \max_{s,t \in [0,T] } \left\{\oline{C} \abs{ \zeta(s) - \zeta(t) } - \frac{|s-t|^2}{2\epsilon} \right\} + C \mathrm{osc}(\zeta,\hat s, \hat t) 
		\le \delta + C \mathrm{osc}(\zeta,\hat s, \hat t),
	\end{align*}
	so that
	\[
		\mathrm{osc}(\zeta,\hat s, \hat t) \le \frac{\delta}{\oline{C}- C} < \frac{\delta}{2} \quad \text{if} \quad \oline{C} > C + 2.
	\] 
	Now, if $\hat n \in \{0, 1, 2, \ldots, N-1 \}$ is the integer satisfying $t_{\hat n} < \hat s \le t_{\hat n + 1}$, then the linearity of $\zeta$ on $[t_{\hat n}, t_{\hat n + 1}]$ implies that
	\[
		|\zeta(\hat s) - \zeta(t_{\hat n })| \le \lambda_0 h < \frac{\delta}{2},
	\]
	and so the triangle inequality yields $(t_{\hat n}, \hat t) \in U_\delta(\zeta)$. This, in turn, means that $(s,\hat t) \in U_\delta(\zeta)$ for all $s \in [t_{\hat n}, \hat s ]$.
	
	We next use the definition of pathwise viscosity solutions to establish the inequality
	\begin{equation}\label{E:equationineq}
		\frac{\hat s - \hat t}{\epsilon} \ge \sigma. 
	\end{equation}
	In view of Lemma \ref{L:distancefunction}(c), for any $x \in \RR^d$, the function
	\[
		y \mapsto v(y,t) + \Phi_\delta(x - y, \hat s, t ; \zeta)
	\]
	attains a global minimum over $\RR^d$. Definition \ref{D:nonsmoothH} and Lemma \ref{L:distancefunction}(d) then imply that
	\[
		t \mapsto \inf_{y \in \RR^d} \left\{ v(y,t) + \Phi_\delta(x - y, \hat s, t; \zeta ) \right\}
	\]
	is nondecreasing on $I := \{ t \in [0,T] : (\hat s, t) \in U_\delta(\zeta) \}$, and therefore, so is
	\[
		\phi(t) := \inf_{x, y \in \RR^d} \left\{ v(y,t) - v_h(x,\hat s) +\Phi_\delta(x - y,\hat s, t ; \zeta) \right\}.
	\]
	Since $\phi(t) + \frac{|\hat s- t|^2}{2\epsilon} + \sigma t$ attains a minimum at $\hat t \in I$, \eqref{E:equationineq} follows.
		
	On the other hand, we obtain a contradiction by using \eqref{A:quantmonotonicity} and \eqref{A:quantconsistency} to show that
	\begin{equation}\label{E:schemeineq}
		\frac{\hat s - \hat t}{\epsilon} \le 0.	 
	\end{equation}
	The first step is to prove that, for each $y \in \RR^d$, the function
	\[
		a(s) := \sup_{x \in \RR^d} \left\{ v_h(x,s) - \Phi_\delta(x - y, s, \hat t; \zeta) \right\} - \frac{\mu(s)}{\delta}
	\]
	satisfies
	\[
		\max_{[t_{\hat n}, t_{\hat n+1}]} a = a(t_{\hat n}).
	\]
	Indeed, if this were not the case, then, for some $s^* \in [t_{\hat n}, t_{ \hat n + 1} ]$ and sufficiently small $\beta > 0$,
	\[
		a(t_{\hat n}) \le a(s^*) - \beta(s^* - t_{\hat n} ).
	\]
	Lemma \ref{L:distancefunction}(c) implies that the supremum in the definition of $a(s^*)$ is attained for some $x^* \in \RR^d$, and so it follows that, for all $x \in \RR^d$,
	\begin{equation} \label{E:applyschemetoineq}
		v_h(x, t_{\hat n}) \le v_h(x^*, s^* ) + \Phi_\delta(x - y, t_{\hat n}, \hat t ; \zeta) - \Phi_\delta(x^* - y, s^*, \hat t; \zeta) - \frac{\mu(s^*) - \mu(t_{\hat n})}{\delta} - \beta(s^* - t_{\hat n}).
	\end{equation}
	In view of \eqref{A:quantmonotonicity} and the fact that $\mathrm{osc}(\zeta,t_{\hat n}, s^*) \le \lambda_0 h$, the operator $S_h(s^*, t_{\hat n}; \zeta)$ is monotone. Applying it to both sides of the inequality \eqref{E:applyschemetoineq}, setting $x = x^*$, rearranging terms, and using \eqref{A:quantconsistency} and Lemma \ref{L:distancefunction}(b) and (d) yield 
	\[
		\frac{\oline{C} h^2}{\delta} + \beta(s^* - t_{\hat n} ) = \frac{\mu(s^*) - \mu(t_{\hat n})}{\delta} + \beta(s^* - t_{\hat n} ) \le C\nor{D^2\Phi}{\oo} h^2 \le \frac{Ch^2}{\delta}.
	\]
	This results in a contradiction as long as $\oline{C} \ge C$.
	
	As a consequence,
	\[
		\psi(s) := \sup_{x, y \in \RR^d} \left\{ v_h(x,s) - v(y, \hat t) - \Phi_\delta(x - y, s, \hat t; \zeta) \right\} - \frac{\mu(s)}{\delta}
	\]
	attains its maximum in $[t_{\hat n}, t_{\hat n + 1}]$ at $t_{\hat n}$, and therefore, because $\psi(s) - \frac{|s - \hat t|^2}{2\epsilon} - \frac{\alpha(s)}{\epsilon}$ attains a maximum at $\hat s$, 
	\[
		\psi(t_{\hat n} ) - \frac{ |t_{\hat n} - \hat t |^2}{2\epsilon} - \frac{\alpha(t_{\hat n})}{\epsilon} \le \psi(\hat s) - \frac{ |\hat s - \hat t |^2}{2\epsilon} - \frac{\alpha(\hat s)}{\epsilon} \le \psi(t_{\hat n}) - \frac{ |\hat s - \hat t |^2}{2\epsilon} - \frac{\alpha(\hat s)}{\epsilon}
	\]
	which, after rearranging terms, yields \eqref{E:schemeineq}. Together with \eqref{E:equationineq}, this establishes \eqref{E:maxPhi}.
	
	{\it Step 2:} The next claim is that, for some $C = C_{L} > 0$,
	\[
		\max_{\{0\} \times [0,T] \cup [0,T] \times \{0\} } \Psi \le C \pars{ \delta + \sqrt{N} h} +\max_{s,t \in [0,T]} \left\{ C\abs{ \zeta(s) - \zeta(t)} - \frac{|s-t|^2}{2\epsilon} \right\}.
	\]
	Assume that $\Psi$ attains its maximum at $(\hat s, \hat t)$, with either $\hat s = 0$ or $\hat t = 0$.
	
	If $\hat s = \hat t = 0$, then Lemmas \ref{L:distancefunction}(c) and \ref{L:solutionbounds} yield $C = C_L > 0$ such that
	\begin{gather*}
		\Psi(0,0)
		= \sup_{x, y \in \RR^d} \left\{ u_0(x) - u_0(y) - \Phi_\delta(x - y, 0,0; \zeta)\right\}\\
		\le \sup_{x, y \in \RR^d} \left\{ L|x - y| - \frac{1}{2(C+1)\delta}|x - y|^2\right \} + C\delta
		\le \pars{ C + \frac{(C+1)L^2}{2} }\delta.
	\end{gather*}
	
	Assume now that $\hat s = 0$. Then, in view of Lemmas \ref{L:distancefunction}(c) and \ref{L:solutionbounds},
	\begin{align*}
		\Psi(0,\hat t)
		&= \sup_{|x - y| \le M \delta} \left\{ u_0(x) - v(y,\hat t) - \Phi_\delta(x - y,0,\hat t; \zeta) \right\} - \frac{\hat t^2}{2\epsilon}\\
		&\le C \delta + \sup_{|x - y| \le M \delta} \left\{ u_0(y) - v(y,\hat t) - \Phi_\delta(x - y, 0, \hat t ; \zeta) \right\} - \frac{\hat t^2}{2\epsilon}\\
		& \le C \delta + \max_{t \in [0,T]} \pars{ C\mathrm{osc}(\zeta,0,t) - \frac{t^2}{2\epsilon}}
		= C \delta + \max_{s,t \in [0,T]} \pars{ C|\zeta(s) - \zeta(t)| - \frac{|s-t|^2}{2\epsilon}}.
	\end{align*}
	
	Finally, if $\hat t = 0$, then Lemma \ref{L:solutionbounds} gives
	\begin{equation*}
		\begin{split}
			\Psi(\hat s, 0)
			&\le \sup_{| x - y| \le M \delta} \left\{ v_h(x,\hat s) - u_0(y) - \Phi_\delta( x - y, \hat s, 0; \zeta) \right\} - \frac{\hat s^2}{2\epsilon} \\
			&\le C \delta + \sup_{x \in \RR^d} \left\{ v_h(x, \hat s) - u_0(x) \right\} - \frac{\hat s^2}{2\epsilon} \\
			&\le C \pars{  \delta + \sqrt{N} h} + \max_{s,t \in [0,T]} \left\{ C\abs{\zeta(s) - \zeta(t)} - \frac{|s-t|^2}{2\epsilon} \right\}.
		\end{split}
	\end{equation*}
	
	{\it Step 3.} Combining the previous two steps and rearranging terms yields, for all $(x,t) \in \RR^d \times [0,T]$,
	\[
		v_h(x,t) - v(x,t) \le \frac{1}{\epsilon} \sum_{n=0}^{N-1} (\Delta t_n)^2 + C \pars{ \delta + \frac{N h^2}{\delta} + \sqrt{N}h } + \max_{s,t \in [0,T]} \left\{ C\abs{ \zeta(s) - \zeta(t)} - \frac{|s-t|^2}{2\epsilon}   \right\}.
	\]
	The inequality is optimized by setting
	\[
		\delta := \max \left\{ C \sqrt{N} h, \max_{s,t \in [0,T]} \pars{ C\abs{ \zeta(s) - \zeta(t)} - \frac{|s-t|^2}{2\epsilon}}  \right\}
	\]
	for a sufficiently large constant $C = C_{L} > 0$, which clearly satisfies \eqref{E:delta}. This finishes the proof of the upper bound for $v_h - v$, and the lower bound is proved similarly.
\end{proof}

\section{Convergence rates} \label{S:regularrates}

In this section, the pathwise estimate from Theorem \ref{T:pathwise} is used to obtain a rate of convergence for schemes approximating solutions of the Hamilton-Jacobi equation 
\begin{equation} \label{E:firstordereq}
	du = \sum_{i=1}^m H^i(Du) \circ dW^i \quad \text{in } \RR^d \times (0,T] \quad \text{and} \quad u(\cdot,0) = u_0 \quad \text{in } \RR^d.
\end{equation}
It will always be assumed, as in Section \ref{S:pathwise}, that 
\begin{equation}\label{A:section6}
	\left\{
	\begin{split}
	&\text{$H$ and $u_0$ satisfy \eqref{A:Hdiffconvex}, \eqref{A:initialcondition}, and \eqref{A:Hlingrowth},}\\[1.2mm]
	&\text{and the scheme operator $S_h$ satisfies \eqref{A:constantcommute}, \eqref{A:translatecommute}, \eqref{A:quantmonotonicity}, and \eqref{A:quantconsistency}}.
	\end{split}
	\right.
\end{equation}

We first examine the setting in which $W$ is a fixed, deterministic path, and then some extensions are presented in the case where $W$ is a Brownian motion. Following Section \ref{S:scheme}, we define $u_h := v_h(\cdot; W_h, \mcl P_h)$, with $v_h$ as in \eqref{E:quantapproxsolution}, for an appropriate family of approximating paths $\{W_h\}_{h > 0}$ and partitions $\{ \mcl P_h\}_{h > 0}$. Let $v$ be the viscosity solution of
\begin{equation}\label{E:intermediatesolution}
	v_t = \sum_{i=1}^m H^i(D v) \dot W_h \quad \text{in } \RR^d \times (0,T] \quad \text{and} \quad v(\cdot, 0) = u_0 \quad \text{in } \RR^d.
\end{equation}
The error $u_h - u$ is then controlled by using Theorem \ref{T:pathwise} and Lemma \ref{L:LSestimate} to estimate respectively the differences $u_h - v$ and $v - u$.

\subsection{A fixed continuous path}

Fix $W \in C([0,T]; \RR^m)$, and let $\omega:[0,\oo) \to [0,\oo)$ be its modulus of continuity. Define $\rho_h$ implicitly by
\begin{equation} \label{E:pathwiseCFL}
	\lambda := \frac{\omega((\rho_h)^{1/2})}{(\rho_h)^{1/2}} < \lambda_0
\end{equation}
and set $\mcl P_h := \left\{ n \rho_h \wedge T \right\}_{n = 0}^N$, where $N$ is the smallest integer for which $N \rho_h \wedge T = T$. 

Recall from subsection \ref{SS:HamiltonJacobi} that taking $W_h$ to be the piecewise linear interpolation of $W$ over the partition $\mcl P_h$ may not, in general, yield a convergent scheme. Instead, we set
\[
	M_h := \left \lfloor (\rho_h)^{-1/2} \right \rfloor
\]
and define $W_h$ by
\begin{equation}\label{E:slowpath}
	W_h(t) := W(kM_h\rho_h) + \pars{ \frac{ W((k+1)M_h\rho_h) - W(kM_h\rho_h)}{M_h\rho_h} } \pars{ t - kM_h\rho_h}
\end{equation}
for $k \in \NN_0$ and $t \in [kM_h\rho_h, (k+1)M_h\rho_h)$.
Observe that the approximating path $W_h$ satisfies \eqref{E:slowderivative} with $\eta_h = (\rho_h)^{1/2}$.

Now set $u_h := v_h(\cdot; W_h, \mcl P_h)$, with $v_h$ as in \eqref{E:quantapproxsolution}, and let $v$ be the solution of \eqref{E:intermediatesolution}.

\begin{theorem} \label{T:fixedcontinuouspath}
	Assume \eqref{A:section6} and let $u_h$ and $u$ be as described above. Then there exists $C = C_{L,\lambda} > 0$ such that
	\begin{equation}\label{E:regularpartitionestimate}
		\sup_{(x,t) \in \RR^d \times [0,T]} \abs{ u_h(x,t) - u(x,t)} \le C(1+ T)\omega((\rho_h)^{1/2}). 
	\end{equation}
\end{theorem}

As an example, assume that $W \in C^\alpha([0,T], \RR^m)$ and set
\[
	\lambda := [W]_{\alpha,T} \frac{(\rho_h)^{(1+\alpha)/2}}{h}.
\]
Then, as long as $\lambda < \lambda_0$, the scheme converges with a rate of order $(\rho_h)^{\alpha/2} \approx h^{\alpha/(1+\alpha)}$.

\begin{proof}[Proof of Theorem \ref{T:fixedcontinuouspath}] First, notice that, in view of \eqref{E:pathwiseCFL}, $W_h$ satisfies \eqref{A:pathrestrict}. In particular, for some $C = C_L > 0$,
\[
	\max_{s,t \in [0,T]} \pars{ C|W_h(s) - W_h(t)| - \frac{|s-t|^2}{2\epsilon}} \le C \lambda_0 h + \max_{n \in \NN_0} \pars{ Cn \lambda_0 h - \frac{n^2 \rho_h^2}{2\epsilon}}
	\le C\lambda_0 h + \frac{(C\lambda_0 h)^2 \epsilon}{2 (\rho_h)^2}.
\]
Theorem \ref{T:pathwise} then gives, for any $\epsilon > 0$,
\begin{align*}
	\max_{(x,t) \in \RR^d \times [0,T]} \abs{ u_h(x,t)  - v(x,t) }
	&\le \frac{N (\rho_h)^2}{\epsilon} +  C \sqrt{N} h + \frac{(C\lambda_0 h)^2 \epsilon}{2 (\rho_h)^2}  \\
	&\le \frac{T \rho_h}{\epsilon} + C \sqrt{T} \frac{h}{\sqrt{\rho_h}} + \frac{(C\lambda_0 h)^2 \epsilon}{2 (\rho_h)^2}. 
\end{align*}
Upon choosing $\epsilon = \sqrt{T} \frac{(\rho_h)^{3/2}}{h}$, this becomes
\begin{equation}\label{E:fixedtimestep}
	\max_{(x,t) \in \RR^d \times [0,T]} \abs{ u_h(x,t)  - v(x,t)} \le C \sqrt{T} \frac{h}{\sqrt{\rho_h}} = C \sqrt{T} \omega((\rho_h)^{1/2}).
\end{equation}
Notice that the error term takes the form $\sqrt{ \frac{h^2}{\rho_h}}$, which is consistent with the discussion in subsection \ref{SS:HamiltonJacobi}.

Lemma \ref{L:LSestimate} then implies that
\[
	\sup_{(x,t) \in \RR^d \times [0,T]} \abs{ u_h(x,t) - u(x,t)} \le C \pars{ \sqrt{T} \omega((\rho_h)^{1/2}) + \omega(M_h \rho_h)},
\]
and the result is proved in view of the choice of $M_h$.
\end{proof}

\subsection{Brownian paths}

For the rest of the paper, we investigate schemes for which $W$ is a standard Brownian motion defined on a probability space $(\Omega, \mcl F, \mcl F_t, \mbf P)$. The expectation and variance with respect to $\mbf P$ are denoted by respectively $\mbf E$ and $\mbf{Var}$. To simplify the presentation, it is assumed that $m = 1$, so that $W$ is one-dimensional, although all three schemes below can be adapted to the case when $m > 1$.

\subsubsection{Regular partitions}

Theorem \ref{T:fixedcontinuouspath} may be applied in this situation by using the fact that oscillations of Brownian paths are controlled by the L\'evy modulus of continuity. More precisely,
\begin{equation}\label{E:Levymod}
	\mbf P \pars{ \limsup_{\delta \to 0} \sup_{\delta \le t \le T-\delta} \frac{\abs{ W(t) - W(t+\delta)}}{\sqrt{2 \delta \abs{ \log \delta } } } = 1} = 1. 
\end{equation}

\begin{theorem} \label{T:regularBrownian}
	Assume \eqref{A:section6}, let $\rho_h$ be defined implicitly by
	\begin{equation} \label{A:BrownianCFL}
		\lambda := \frac{(\rho_h)^{3/4} \abs{\log \rho_h }^{1/2} }{h} < \lambda_0,
	\end{equation}
	and let $u_h$, $\mcl P_h$, and $W_h$ be as in the previous subsection. Then there exists a deterministic constant $C = C_{L, \lambda} > 0$ such that
	\[
		\mbf P \pars{ \limsup_{ h \to 0}  \sup_{(x,t) \in \RR^d \times [0,T]} \frac{ \abs{ u_h(x,t) - u(x,t) } }{ (\rho_h)^{1/4} \abs{\log \rho_h}^{1/2} } \le C(1+T) } = 1.
	\]
\end{theorem}

\begin{proof}
Define $M_h := \lfloor (\rho_h)^{-1/2} \rfloor$ and $K_h := \lfloor T/(M_h\rho_h) \rfloor$. The definitions of $W_h$ and $\lambda$ give
	\begin{align*}
		\max_{n=0,1,2\ldots, N - 1 } &\frac{ \abs{ W_h(n\rho_h) - W_h((n+1)\rho_h)} }{h}
		=
		\max_{k=0,1,2 \ldots, K_h} \frac{ \abs{ W(kM_h \rho_h) - W((k+1)M_h\rho_h} }{M_h h} \\
		&=
		\lambda \max_{k=0,1,2 \ldots, K_h} \frac{ \abs{ W(kM_h \rho_h ) - W((k+1)M_h \rho_h) } }{M_h(\rho_h)^{3/4} \abs{ \log \rho_h }^{1/2} } 
		\le \lambda \frac{ \max_{|s-t| \le (\rho_h)^{1/2}} \abs{W(s) - W(t) } }{ (\rho_h)^{1/4}(1-(\rho_h)^{1/2}) \abs{ \log \rho_h }^{1/2} }.
	\end{align*}
	Therefore, in view of \eqref{E:Levymod}, for any $\delta > 0$,
	\[
		\mbf P \pars{ \max_{n=0,1,2, \ldots, N -1} \frac{ \abs{ W_h(n\rho_h) - W_h((n+1)\rho_h)} }{h}
		\le \lambda \frac{1+\delta}{1 - (\rho_h)^{1/2}} \quad \text{for sufficiently small $h$}} = 1.
	\]
	Taking $\delta \in (0, \lambda_0/\lambda - 1)$, this implies that
	\[
		\mbf P \pars{ \limsup_{h \to 0} \max_{n = 0, 1, 2, \ldots, N - 1} \frac{ \abs{ W_h(n \rho_h) - W_h((n+1)\rho_h )} }{h} < \lambda_0} = 1,
	\]
	so that, for some $h_0 > 0$, 
	\[
		\mbf P \pars{ \abs{ W_h(n\rho_h) - W_h((n+1)\rho_h) } \le \lambda_0 h \quad \text{for all } 0 < h < h_0 \text{ and } n = 0, 1, 2, \ldots, N_h - 1 } = 1.
	\]
	Shrinking $h_0$, if necessary, it may be concluded from \eqref{E:regularpartitionestimate} and \eqref{E:Levymod} that
	\[
		\mbf P \pars{ \sup_{ (x,t) \in \RR^d \times [0,T] } \abs{ u_h(x,t) - u(x,t) } \le CT(\rho_h)^{1/4} \abs{ \log \rho_h}^{1/2} \quad \text{for all } 0 < h < h_0 } = 1.
	\]
\end{proof}

Observe that \eqref{A:BrownianCFL} implies that $\lim_{h\to 0} \frac{\log \rho_h}{\log h} = \frac{4}{3}$, so that the convergence rate in Theorem \ref{T:regularBrownian} can be rewritten as
\begin{equation} \label{E:intermsofh}
	 \limsup_{ h \to 0}  \sup_{(x,t) \in \RR^d \times [0,T]} \frac{ \abs{ u_h(x,t) - u(x,t) } }{ h^{1/3} \abs{\log h}^{1/3}  } \le C(1+T).
\end{equation}

\subsubsection{Random partitions}

For the next scheme, the partitions $\mcl P_h$ are defined using a sequence of stopping times adapted to the filtration $\mcl F_t$ of the Brownian motion $W$. By choosing the stopping times carefully to control the maximal oscillations of the Brownian paths, it is possible to recover the error estimate from Theorem \ref{T:regularBrownian}.


For $h > 0$, define $\eta_h := h^{1/3} \abs{ \log h}^{-2/3}$, set $T_0 = T_0(h) := 0$, and, for $k \in \NN_0$,
\[
	T_{k+1} = T_{k+1}(h) := \inf \left\{ t > T_k(h) : \mathrm{osc}\pars{ W, T_k(h), t} > \eta_h \right\} \quad \text{and} \quad \tau_{k+1} =  \tau_{k+1}(h) := T_{k+1}(h) - T_k(h).
\]
Observe that $\{T_k\}_{k = 0}^\oo$ is an increasing sequence of stopping times, and, for each fixed $k$, $h \to T_k(h)$ decreases as $h \to 0$. Therefore, by the strong Markov property for Brownian motion, for each fixed $h$, $\{\tau_k(h) \}_{k=1}^\oo$ is a collection of independent, identically distributed random variables. As a result, for any integer $\ell > 0$, there exists a constant $c_\ell > 0$ such that, for all $k$,
\[
	\mbf E[ \tau_k(h)^\ell] = c_\ell (\eta_h)^{2 \ell}.
\]
Indeed, it is well known that the first exit time of a Brownian motion from a bounded interval has finite moments of any order. The exact formula follows from the scaling properties of Brownian motion, so that
\[
	c_\ell := \mbf E \left[ \inf \left\{ t > 0 : \mathrm{osc}\pars{ W,0,t^{1/\ell} } > 1 \right\} \right].
\]

Let $W_h$ be the piecewise interpolation of $W$ over the partition $\left\{ 0 = T_0(h) < T_1(h) < T_2(h) < \cdots \right\}$. That is,
\[
	W_h(t) := W(T_k(h)) + \frac{ W(T_{k+1}(h)) - W(T_k(h))}{\tau_{k+1}(h)} (t - T_k(h)) \quad \text{whenever} \quad T_k(h) \le t < T_{k+1}(h).
\]
Define 
\[
	M_h := \left\lceil \frac{\eta_h}{\lambda_0 h}\right\rceil = \left \lceil (\lambda_0 h^{2/3} |\log h|^{2/3} )^{-1} \right\rceil,
\]
$t_0 = t_0(h) := 0$, and, whenever $k = 0, 1, 2, \ldots$ and $kM_h \le n < (k+1)M_h$,
\[
	t_n = t_n(h) := T_k(h) + (n - kM_h) \frac{\tau_{k+1}(h)}{M_h} \quad \text{and} \quad \Delta t_n = \Delta t_n(h) := t_{n+1}(h) - t_n(h) = \frac{\tau_{k+1}(h)}{M_h}.
\]
Also set
\[
	K_h := \sup \left\{ k \in \NN_0: T_k(h) \le T \right\} \quad \text{and} \quad N_h := \sup \left\{ n \in \NN_0 : t_n(h) \le T \right\},
\]
and note that $h \mapsto K_h$ increases as $h \to 0$.

We have defined the path $W_h$, which is piecewise linear over the partition
\[
	\mcl P_h := \left\{ 0 = t_0(h) < t_1(h) < t_2(h) < \cdots < t_{N_h}(h) \le T \right\},
\]
in such a way that \eqref{A:pathrestrict} holds for $\zeta = W_h$. Indeed, if $n = 0, 1, 2, \ldots, N-1$ and $k$ is such that $kM_h \le t_n < t_{n+1} \le (k+1)M_h$, then
\[
	\abs{ W_h(t_{n+1}) - W_h(t_n)} = \frac{ \abs{ W(T_{k+1}) - W(T_k)} }{M_h} \le \lambda_0 h.
\]

Finally, set $u_h := v_h(\cdot; W_h, \mcl P_h)$ and let $u$ be the stochastic viscosity solution of \eqref{E:firstordereq}.

\begin{theorem} \label{T:randompart}
	Assume \eqref{A:section6}, and let $u_h$ and $u$ be as described above. Then there exists a deterministic constant $C = C_{L} > 0$ such that
	\[
		\mbf P \pars{\limsup_{h \to 0} \max_{(x,t) \in \RR^d \times [0,T]} \frac{ \abs{ u_h(x,t) - u(x,t) }}{h^{1/3} \abs{ \log h}^{1/3} } \le C(1+T)} = 1.
	\]
\end{theorem}

We proceed with a series of lemmas that indicate how to control the various terms appearing in the estimate from Theorem \ref{T:pathwise}.

\begin{lemma} \label{L:maxstoppingtime}
	\[
		\mbf P \pars{ \limsup_{h \to 0} K_h \eta_h^2 \le \frac{T}{c_1} } = 1.
	\]
\end{lemma}

\begin{proof}
	Fix $\alpha$ and $\beta$ such that $1 < \beta^{2/3} < \alpha$, and define $h_m := \beta^{-m}$. Note that
	\[
		\lim_{m \to \oo} \frac{\eta_{h_{m+1}}}{\eta_{h_m}} = \frac{1}{\beta^{1/3}}.
	\]
	The monotonicity of $K_h$ and $\eta_h$ implies that
	\begin{equation}
		\mbf P \pars{ \sup_{h_{m+1} \le h < h_m} K_h \eta_h^2 > \frac{\alpha T}{c_1} }
		\le \mbf P \pars{ K_{h_{m+1}} > \frac{\alpha T}{c_1 \eta_{h_m}^2 } }. \label{E:fixh}
	\end{equation}
	Set
	\[
		k_m := \left\lceil \frac{\alpha T}{c_1 \eta_{h_m}^2} \right\rceil,
	\]
	so that
	\[
		k_m c_1 \eta_{h_{m+1}}^2 \ge \alpha T \pars{ \frac{\eta_{h_{m+1}}}{\eta_{h_m}} }^2 \xrightarrow{m \to \oo} \alpha \beta^{-2/3}T > T,
	\]
	and therefore, for any fixed $\gamma > 0$ and all sufficiently large $m$, $k_m c_1 \eta_{h_{m+1}}^2 \ge (1 + \gamma)T$.

	Define $\sigma^2 := c_2 - c_1^2$, so that $\mbf{Var}(\tau_k(h)) = \sigma^2 \eta_h^4$ for all $k$ and $h$. Continuing \eqref{E:fixh} and applying Markov's inequality yields, for some fixed positive constant $C > 0$ and for all sufficiently large $m$,
	\begin{align*}
		\mbf P \pars{ K_{h_{m+1}} > \frac{\alpha T}{c_1 \eta_{h_m}^2 } } 
		&= \mbf P \pars{ \sum_{k=1}^{k_m} \tau_k(h_{m+1}) \le T }
		\le \mbf P \pars{ \sum_{k=1}^{k_m} \pars{ \tau_k(h_{m+1}) - c_1 \eta_{h_{m+1}}^2 } \le - \gamma T} \\
		&\le \frac{ k_m \sigma^2 \eta_{h_{m+1}}^4 }{ \gamma^2 T^2 }
		\le C \beta^{-2m/3}.
	\end{align*}
	The Borel-Cantelli lemma applied to the events
	\[
		E_m := \left\{ \sup_{h_{m+1} \le h < h_m} K_h \eta_h^2 > \frac{\alpha T}{c_1} \right\}
	\]
	gives
	\[
		\mbf P \pars{ \limsup_{h \to 0} K_h \eta_h^2 > \frac{\alpha T}{c_1} } = \mbf P \pars{ \limsup_{m \to \oo} E_m} = 0,
	\]
	and we may conclude upon sending $\alpha \to 1^+$.
\end{proof}

\begin{lemma} \label{L:epsiloninvterm}
	\[
		\mbf P \pars{ \limsup_{h \to 0} \frac{1}{h \eta_h} \sum_{n=0}^{N-1} \pars{ \Delta t_n}^2 \le \frac{T\lambda_0 c_2}{c_1} }  = 1.
	\]
\end{lemma}

\begin{proof}
	Fix $\alpha$ and $\beta$ satisfying $1 < \beta^{7/3} < \alpha$ and set $h_m := \beta^{-m}$. If, for some $m$, $h_{m+1} \le h < h_m$, then
	\[
		\sum_{n=0}^{N_h -1} \pars{\Delta t_n(h)}^2 \le \sum_{k=1}^{K_h+1} M_h \pars{ \frac{\tau_{k}(h)}{M_h} }^2 \le \lambda_0 \frac{h_m}{\eta_{h_{m+1}}} \sum_{k=1}^{K_{h_{m+1}}+1} \tau_{k}(h_m)^2.
	\]
	Fix $m_0 \in \NN$ and define the event
	\[
		E_{m_0} := \left\{ K_{h_{m+1}} + 1 \le \hat K_m := \left\lceil \frac{\alpha T}{c_1 \eta_{h_{m+1}}^2} \right\rceil \quad \text{for all} \quad m \ge m_0 \right\}.
	\]
	In view of Lemma \ref{L:maxstoppingtime}, $\lim_{m_0 \to \oo} \mbf P\pars{ E_{m_0}} = 1$.
	
	Now, for any $m \ge m_0$,
	\begin{align*}
		\mbf P & \left( \left\{ \sup_{h_{m+1} \le h < h_m} \frac{1}{h \eta_h} \sum_{n=0}^{N_h-1} \pars{\Delta t_n(h)}^2 > \frac{\alpha^2 T \lambda_0 c_2}{c_1} \right\} \cap E_{m_0} \right)\\
		&\le \mbf P \pars{ \sum_{k=1}^{\hat K_m} \tau_{k}(h_m)^2 > \frac{\alpha^2 T c_2 h_{m+1}( \eta_{h_{m+1}})^2}{c_1 h_m} }\\
		&= \mbf P \pars{ \sum_{k=1}^{\hat K_m} \pars{ \tau_{k}(h_m)^2 - c_2 \eta_{h_m}^4} > \frac{\alpha^2 T c_2 h_{m+1} (\eta_{h_{m+1}})^2}{c_1 h_m} - \hat K_m c_2 \eta_{h_m}^4}\\
		&\le \mbf P \pars{  \sum_{k=1}^{\hat K_m} \pars{ \tau_{k}(h_m)^2 - c_2 \eta_{h_m}^4} > \frac{\alpha T c_2 \eta_{h_m}^2}{c_1} \pars{ \frac{\alpha h_{m+1} (\eta_{h_{m+1} })^2 }{h_m (\eta_{h_m})^2} - \frac{\eta_{h_m}^2}{\eta_{h_{m+1}}^2} } - c_2 \eta_{h_m}^4}.
	\end{align*}
	Since
	\[
		\lim_{m \to \oo} \pars{ \frac{\alpha h_{m+1} \eta_{h_{m+1}}^2}{h_m \eta_{h_m}^2} - \frac{\eta_{h_m}^2}{\eta_{h_{m+1}}^2}} = \frac{\alpha}{\beta^{5/3}} - \beta^{2/3} > 0,
	\]
	it follows that, for some fixed $\gamma > 0$, all sufficiently large $m_0$, and all $m > m_0$,
	\[
		\mbf P \left( \left\{ \sup_{h_{m+1} \le h < h_m} \frac{1}{h \eta_h} \sum_{n=0}^{N_h-1} \pars{\Delta t_n(h)}^2 > \frac{\alpha^2 T \lambda_0 c_2}{c_1} \right\} \cap E_{m_0} \right)
		\le \mbf P \pars{ \sum_{k=1}^{\hat K_m} \pars{ \tau_{k}(h_m)^2 - c_2 \eta_{h_m}^4} > \gamma \eta_{h_m}^2}.
	\]
	Set $\sigma^2 := c_4 - c_2^2 > 0$. Then Markov's inequality gives, for some constant $C > 0$ independent of $m$,
	\begin{align*}
		\mbf P & \left( \left\{ \sup_{h_{m+1} \le h < h_m} \frac{1}{h \eta_h} \sum_{n=0}^{N_h - 1} \pars{\Delta t_n(h)}^2 > \frac{\alpha^2 T \lambda_0 c_2}{c_1} \right\} \cap E_{m_0} \right)\\
		&\le \frac{ \hat K_m \sigma^2 \eta_{h_m}^4}{\gamma^2}\le C \eta_{h_m}^2 \le C \beta^{-2m/3}.
	\end{align*}
	An application of the Borel-Cantelli lemma for the events
	\[
		 \left\{ \sup_{h_{m+1} \le h < h_m} \frac{1}{h \eta_h} \sum_{n=0}^{N_h - 1} \pars{\Delta t_n(h)}^2 > \frac{\alpha^2 T \lambda_0 c_2}{c_1} \right\} \cap E_{m_0}
	\]
	yields
	\[
		\mbf P \pars{ \left\{ \limsup_{h \to 0}  \frac{1}{h \eta_h} \sum_{n=0}^{N_h-1} \pars{\Delta t_n(h)}^2 > \frac{\alpha^2 T \lambda_0 c_2}{c_1} \right\} \cap E_{m_0} } = 0.
	\]
	Sending $m_0 \to \oo$ and then $\alpha \to 1^+$ finishes the proof.
\end{proof}

\begin{lemma} \label{L:penalize}
	For any deterministic constant $C> 0$,
	\[
		\mbf P \pars{ \limsup_{\epsilon \to 0} \frac{ \max_{s,t \in [0,T]} \left\{ C \abs{W(s) - W(t)} - \frac{|s-t|^2}{2 \epsilon} \right\} }{ \epsilon^{1/3} \abs{ \log \epsilon }^{2/3} } \le \frac{4C^{4/3}}{3^{2/3}} } =1.
	\]
\end{lemma}

\begin{proof}
	Let $1 < \beta < \alpha$. If, for some $\delta > 0$,
	\begin{equation} \label{E:Brownianincrements}
		\mathrm{osc}(W, k\delta, (k+1)\delta) \le \sqrt{2\beta} \delta^{1/2} \abs{ \log \delta}^{1/2} \quad \text{for all} \quad k = 0, 1, 2, \ldots, \left\lceil \frac{T}{\delta} \right\rceil,
	\end{equation}
	then
	\begin{align*}
		\max_{s,t \in [0,T]} \left\{ C |W(s) - W(t)| - \frac{|s-t|^2}{2\epsilon} \right\}
		&\le \sqrt{2\beta} C \delta^{1/2} \abs{ \log \delta}^{1/2} + \max_{n \in \NN_0} \left\{ \sqrt{2\beta} C \delta^{1/2} \abs{ \log \delta}^{1/2}n - \frac{n^2 \delta^2}{2\epsilon} \right\}\\
		&\le \sqrt{2\beta} C \delta^{1/2} \abs{ \log \delta}^{1/2} + \beta C^2 \frac{\abs{ \log \delta} }{\delta} \epsilon.
	\end{align*}
	Taking $\delta := C^{2/3} 3^{-1/3} \beta^{1/3} \epsilon^{2/3} \abs{ \log \epsilon}^{1/3}$ yields, for some deterministic function $c(\epsilon) \xrightarrow{\epsilon \to 0} 0$,
	\[
		\max_{s,t \in [0,T]} \left\{ C |W(s) - W(t)| - \frac{|s-t|^2}{2\epsilon} \right\} \le \frac{4 C^{4/3}}{3^{2/3}} \beta^{2/3} \epsilon^{1/3} |\log \epsilon|^{2/3}(1 + c(\epsilon)),
	\]
	and, therefore, if $\epsilon$ is sufficiently small,
	\[
		\max_{s,t \in [0,T]} \left\{ C |W(s) - W(t)| - \frac{|s-t|^2}{2\epsilon} \right\} \le \frac{4 C^{4/3}}{3^{2/3}} \alpha^{2/3} \epsilon^{1/3} |\log \epsilon|^{2/3}.
	\]
	Define
	\[
		\epsilon_m := \alpha^{-m} \quad \text{and} \quad \delta_m := \frac{ C^{2/3} \beta^{1/3} \epsilon_m^{2/3} \abs{ \log \epsilon_m}^{1/3} }{3^{1/3}},
	\]
	and note that
	\[
		\lim_{m \to \oo} \frac{\epsilon_m^{1/3} |\log \epsilon_m|^{2/3} }{\epsilon_{m+1}^{1/3} |\log \epsilon_{m+1}|^{2/3}} = \alpha^{1/3}.
	\]
	It follows that, for sufficiently large $m$, 
	\begin{align*}
		&\mbf P \pars{ \sup_{\epsilon_{m+1} \le \epsilon < \epsilon_m} \frac{ \max_{s,t \in [0,T]} \left\{ C \abs{ W(s) - W(t)} - \frac{|s-t|^2}{2\epsilon}  \right\} }{\epsilon^{1/3} \abs{ \log \epsilon}^{2/3}} > \frac{4 C^{4/3} \alpha}{3^{2/3}} }\\
		&\le \mbf P \pars{ \max_{s,t \in [0,T]} \left\{ C |W(s) - W(t)| - \frac{|s-t|^2}{2\epsilon_m} \right\} > \frac{4 C^{4/3} \alpha \epsilon_{m+1}^{1/3} |\log \epsilon_{m+1}|^{2/3} }{3^{2/3}}  }\\
		&\le \mbf P \pars{ \max_{s,t \in [0,T]} \left\{ C |W(s) - W(t)| - \frac{|s-t|^2}{2\epsilon_m} \right\} > \frac{ 4 C^{4/3} \alpha^{2/3} \epsilon_{m}^{1/3} |\log \epsilon_{m}|^{2/3}}{3^{2/3}} }\\
		&\le \mbf P \pars{ \mathrm{osc}(W, k\delta_m, (k+1)\delta_m) > \sqrt{2\beta} \delta_m^{1/2} \abs{ \log \delta_m}^{1/2}  \text{ for some }  k = 0, 1, 2, \ldots, \left\lceil \frac{T}{\delta_m} \right\rceil}\\
		&\le \left \lceil \frac{T}{\delta_m} \right \rceil \mbf P \pars{ \max_{[0,1]} W - \min_{[0,1]} W > \sqrt{2\beta} \abs{ \log \delta_m}^{1/2}}\\
		&\le 2 \left \lceil \frac{T}{\delta_m} \right \rceil \mbf P \pars{ \max_{[0,1]} W > \sqrt{2\beta} \abs{\log \delta_m}^{1/2} }\\
		&\le CT \delta_m^{\beta - 1} \le CT \alpha^{-\gamma m} \quad \text{for} \quad \gamma = \frac{2}{3} (\beta - 1) > 0.
	\end{align*}
	The symmetry and scaling properties of Brownian motion, as well as the reflection principle, were all used above. In particular, since the processes
	\[
		\left\{ t \mapsto \max_{s \in [0,t]} W(s) - W(t) \right\} \quad \text{and} \quad |W|
	\]
	are identically distributed, so are the random variables
	\[
		\max_{[0,1]} W - \min_{[0,1]} W \quad \text{and} \quad \max_{[0,1]} |W| = \max \left\{ \max_{[0,1]} W, - \min_{[0,1]} W \right\}.
	\]

	The Borel-Cantelli lemma implies that
	\[
		\mbf P \pars{ \limsup_{\epsilon \to 0} \frac{ \max_{s,t \in [0,T]} \left\{ C \abs{W(s) - W(t)} - \frac{|s-t|^2}{2 \epsilon} \right\} }{ \epsilon^{1/3} \abs{ \log \epsilon }^{2/3} } > \frac{4 C^{4/3} \alpha}{3^{2/3}} } =0,
	\]
	and sending $\alpha \to 1^+$ gives the result.
\end{proof}

\begin{proof}[Proof of Theorem \ref{T:randompart}]
	Let $v$ be the solution of \eqref{E:intermediatesolution}. Then Lemma \ref{L:LSestimate} gives, for some $C = C_{L} > 0$,
	\[
		\sup_{(x,t) \in \RR^d \times [0,T]} \abs{ v(x,t) - u(x,t)} \le C \abs{ W(t) - W_h(t)} \le C \max_{k = 0, 1, 2, \ldots, K_h} \mathrm{osc}\pars{ W, T_k(h), T_{k+1}(h)} \le C \frac{h^{1/3}}{ \abs{ \log h}^{2/3}}.
	\]
	
	Next, define $\epsilon_h := \frac{h}{ \abs{ \log h}}$ and recall the pathwise estimate from Theorem \ref{T:pathwise}:
	\begin{gather*}
		\max_{ (x,t) \in \RR^d \times [0,T] } \abs{u_h(x,t) - v(x,t)} \\
		\le \frac{1}{\epsilon_h} \sum_{n=0}^{N_h-1} (\Delta t_n(h))^2 +  C \sqrt{N_h} h + \max_{s,t \in [0,T]} \left\{ C \abs{W_h(s) - W_h(t)} - \frac{|s-t|^2}{2\epsilon_h}  \right\}.
	\end{gather*}
	
	From the definitions of $N_h$, $M_h$, and $K_h$, and from Lemma \ref{L:maxstoppingtime}, it follows that, for some $C = C_{L} > 0$, with probability one, for all sufficiently small $h$,
	\[
		C\sqrt{N_h} h \le C \sqrt{(K_h+1)M_h} h \le CT h^{1/3} \abs{ \log h}^{1/3}.
	\]
	
	Meanwhile, Lemma \ref{L:epsiloninvterm} yields $C = C_{L} > 0$ such that, with probability one, for all sufficiently small $h$,
	\[
		\frac{1}{\epsilon_h} \sum_{n=0}^{N_h-1} (\Delta t_n(h))^2
		= \frac{1}{h \eta_h} \sum_{n = 0}^{N_h-1} (\Delta t_n)^2 \cdot h^{1/3} \abs{ \log h}^{1/3}
		\le CT h^{1/3} \abs{ \log h }^{1/3}.
	\]
	
	In view of the definition of $W_h$,
	\[
		\max_{ 0 \le t \le T} \abs{ W_h(t) - W(t)} \le \max_{k = 0, 1, 2, \ldots, K_h} \mathrm{osc}\pars{ W, T_k(h),T_{k+1}(h)} \le \eta_h,
	\]
	so that, with probability one, for all $h$,
	\[
		\frac{ \max_{s,t \in [0,T]} \left\{ C \abs{ W_h(s) - W_h(t)} - \frac{|s-t|^2}{2 \epsilon_h} \right\} }{h^{1/3} \abs{ \log h}^{1/3} }
		\le \frac{ \max_{s,t \in [0,T]} \left\{ C \abs{W(s) - W(t)} - \frac{|s-t|^2}{2 \epsilon_h} \right\} }{ h^{1/3} \abs{ \log h}^{1/3} }
		+ \frac{C}{\abs{ \log h}},
	\]
	while Lemma \ref{L:penalize} implies that, with probability one and for all sufficiently small $h$,
	\[
		\max_{s,t \in [0,T]} \left\{ C \abs{W(s) - W(t)} - \frac{|s-t|^2}{2\epsilon_h} \right\} 
		\le C \epsilon_h^{1/3} |\log \epsilon_h|^{2/3} \le C h^{1/3}\abs{ \log h}^{1/3}.
	\]
	Combining all terms in the estimate finishes the proof.
\end{proof}

\subsubsection{Scaled random walks and convergence in law} 
The point of view for the preceding approximations was pathwise; that is, the schemes converged for $\mbf P$-almost every sample path of Brownian motion. Here, the strategy is to use independent Rademacher random variables to build an object that converges to the solution of \eqref{E:eq} in distribution. This construction has the advantage that it is simple to implement numerically.

Fix a probability space $(\mcl A, \mcl G, \mbb P)$, not necessarily related to $(\Omega, \mcl F, \mbf P)$, and let $\{\xi_n \}_{n=1}^\oo: \mcl A \to \{-1,1\}$ be independent and identically distributed with
\[
	\mbb P( \xi_n = 1) = \mbb P(\xi_n = -1) = \frac{1}{2}.
\]
Define $\rho_h$ by
\[
	\lambda := \frac{(\rho_h)^{3/4}}{h} \le \lambda_0,
\]
and, as before, set $M_h = \lfloor (\rho_h)^{-1/2} \rfloor$, $\mcl P_h := \{t_n\}_{n =0}^N = \left\{ n \rho_h \wedge T \right\}_{n = 0}^{N}$, $W_h(0) = 0$ and, for $k \in \NN_0$ and $t \in [kM_h \rho_h , (k+1)M_h \rho_h)$,
\[
	W_h(t) := W_h(kM_h \rho_h) + \frac{ \xi_k }{\sqrt{M_h \rho_h} }(t - kM_h \rho_h).
\]
The path $W_h$ is a parabolically scaled simple random walk, and therefore, as is well known, as $h \to 0$, $W_h$ converges to the Wiener process $B$ in distribution. More precisely, if $\mu$ is the Wiener measure on $X := C([0,T], \RR)$ and $\mu_h$ is the probability measure on $X$ induced by $W_h$, then $\mu_h$ converges weakly to $\mu$ as $h \to 0$, that is, for any bounded continuous function $\phi: X \to \RR$,
\[
	\lim_{h \to 0} \int_X \phi \; d\mu_h = \int_X \phi\;d\mu.
\]

Define $u_h := v_h(\cdot; W_h, \mcl P_h) \in BUC(\RR^d \times [0,T])$ and let $v \in BUC(\RR^d \times [0,T])$ be the solution of \eqref{E:intermediatesolution}.
\begin{theorem} \label{T:law}
	Assume \eqref{A:section6} and let $u_h$ and $u$ be as described above. As $h \to 0$, $u_h$ converges to $u$ in distribution.
\end{theorem}

\begin{proof}
Observe first that
\[
	\abs{ W_h(t_{n+1}) - W_h(t_n)} = \sqrt{\frac{\rho_h}{M_h}} \le (\rho_h)^{3/4} \le \lambda_0 h,
\]
so that $W_h$ satisfies \eqref{A:pathrestrict}. Then \eqref{E:regularpartitionestimate} becomes, for some $C = C_{L, \lambda} > 0$,
\begin{equation}\label{E:inlawestimate}
	\max_{(x,t) \in \RR^d \times [0,T]} \abs{ u_h(x,t) - v(x, t)} \le C(1+T) (\rho_h)^{1/4} = C(1+T) h^{1/3}. 
\end{equation}
Lemma \ref{L:LSestimate} implies that the map 
\[
	S: X = C([0,T], \RR) \ni \zeta \mapsto  v \in BUC(\RR^d \times [0,T]) =: Y,
\]
where $v$ is the solution of \eqref{E:zetaeq}, is uniformly continuous. Let $\tilde \nu _h$ and $\nu$ be the push-forwards by $S$ of respectively $\mu_h$ and $\mu$, that is, for any measurable $\psi: Y \to \RR$,
\[
	\int_Y \psi \; d\nu = \int_X \psi \circ S \; d\mu,
\]
with the analogous relation holding for $\nu_h$ and $\nu$. It is clear that $\tilde \nu_h$ converges weakly to $\nu$. On the other hand, if $\nu_h$ is the measure on $BUC(\RR^d \times [0,T])$ induced by $u_h$, then \eqref{E:inlawestimate} and Slutzky's theorem imply that, as $h \to 0$, $\nu_h$ converges weakly to $\nu$.
\end{proof}

\end{document}